\documentclass[12pt,leqno,twoside]{amsart}

\usepackage[latin1]{inputenc}
\usepackage[T1]{fontenc}
\usepackage[colorlinks=true, pdfstartview=FitV, linkcolor=blue, citecolor=blue, urlcolor=blue]{hyperref}
\usepackage{amstext,amsmath,amscd, bezier,indentfirst,amsthm,amsgen,enumerate, geometry}
\usepackage[all,knot,arc,import,poly]{xy}
\usepackage{amsfonts,color}
\usepackage{amssymb,mathtools}

\usepackage{latexsym}
\usepackage{epsfig}
\usepackage{graphicx}
\usepackage{srcltx}
\usepackage{enumitem}
\usepackage{amsmath}
\usepackage{amssymb}
\usepackage{tikz,float}

\newcommand{\fin}{\hspace*{\fill}$\square$\vspace*{2mm}}

\theoremstyle{plain}
\newtheorem{theorem}{Theorem}[section]
\newtheorem{lemma}[theorem]{Lemma}
\newtheorem{proposition}[theorem]{Proposition}
\newtheorem{propdef}[theorem]{Proposition/Definition}
\newtheorem{corollary}[theorem]{Corollary}

\theoremstyle{definition}
\newtheorem{definition}[theorem]{Definition}

\theoremstyle{remark}
\newtheorem{remark}[theorem]{\sc Remark}
\newtheorem{example}[theorem]{\sc Example}

\newcommand{\dist}{{\rm{dist}}}

\newcommand{\Cone}{{\rm{Cone\hspace{2pt}}}}

\newcommand{\Sing}{{\rm{Sing\hspace{2pt}}}}
\newcommand{\Atyp}{{\rm{Atyp\hspace{2pt}}}}

\newcommand{\Jac}{{\rm{Jac\hspace{2pt}}}}

\newcommand{\im}{{\rm{Im\hspace{2pt}}}}
\newcommand{\Sp}{{\rm{Sp}}}
\newcommand{\Va}{{\rm{Va}}}

\newcommand{\mult}{{\rm{mult}}}

\newcommand{\ind}{\mathop{\rm{ind}}}
\newcommand{\lot}{\mathop{\rm{l.o.t.}}}
\newcommand{\ord}{\mathop{\rm{ord}}}
\newcommand{\ori}{\mathop{\rm{or}}}

\newcommand{\ity}{{\infty}}

\newcommand{\e}{\varepsilon}
\newcommand{\m}{\setminus}

\newcommand{\red}{{\mathrm{red}}}

\def\bC{{\mathbb C}}
\def\bK{{\mathbb K}}

\def\bP{{\mathbb P}}

\def\bR{{\mathbb R}}

\def\cA{{\mathcal A}}

\def\cC{{\mathcal C}}

\def\cS{{\mathcal S}}

\def\Atyp{{\mathrm{Atyp\hspace{1pt}}}}
\def\cL{{\mathcal L}}

\def\ity{\infty}
\def\grad{\mathrm{{grad\hspace{1pt}}}}
\def\ind{\mathrm{ind}}
\def\e{{\varepsilon}}

\def\m{{\setminus}}

\begin{document}

\title[Global index of  real polynomials]{Global index of  real polynomials}

  \author{\sc Gabriel E. Monsalve}  

\address{Universidade de S\~{a}o Paulo-ICMC,  S\~{a}o Carlos, Brasil, 
 and Univ. Lille, CNRS,  UMR 8524 - Laboratoire Paul Painlev\'e, F-59000 Lille, France. }

\email{esteban.monsalve@usp.br}

\author{\sc Mihai Tib\u{a}r}

\address{Univ. Lille, CNRS, UMR 8524 - Laboratoire Paul Painlev\'e, F-59000 Lille, France.}

\email{mihai-marius.tibar@univ-lille.fr}

\thanks{G. Monsalve was supported by the FAPESP grants  2019/24377-2  and  2020/14111-2. \\
M. Tib\u ar acknowledges the support of the Labex CEMPI grant (ANR-11-LABX-0007-01), and the support of the project
 ``Singularities and Applications'' -  CF 132/31.07.2023 funded by the European Union - NextGenerationEU - through Romania's National Recovery and Resilience Plan.}

\subjclass[2010]{57R70, 58K60, 55R55, 14H20}


\keywords{polynomial functions, atypical fibres,  index of the gradient}

\begin{abstract}
We develop two methods for expressing the global index of  the gradient  of a 2 variable polynomial function $f$: in terms of the atypical fibres of $f$, and in terms of the clusters of Milnor arcs at infinity. These allow us  to derive upper bounds for the global index, in particular refining the one that was found by Durfee in terms of the degree of $f$.
\end{abstract}
\maketitle

\section{Introduction}

 The index of a vector field with isolated zeroes enters in the celebrated Poincar\'e-Hopf theorem which holds for compact manifolds and has been extended in various directions, either in real or in  complex geometry. 
The index  of the gradient vector field $\grad f$ along the boundary of a large disk which includes all the singularities of a polynomial function $f : \bR^{2} \to\bR$ of degree $d\ge 2$ is the ``global index'', or the  ``index at infinity'' of $f$, denoted $\ind_\ity f$. 
The study of this index at infinity originates, as far as we know,  in Durfee's paper \cite{Dur2}, and is addressed in several other papers, see e.g. \cite{Se}, \cite{Se2}, \cite{Gw1}.

  In the local setting, 
 Arnold's index theorem  \cite{Ar} asserts that the index at an isolated singular point $p\in C$ of a real plane curve $C :=\{g=0\}$ with  $r$ branches satisfies the equality  $\ind_{p} (\grad g )= 1-r$.   In the complex setting, for a holomorphic function germ $h$ of $n\ge 2$ variables and with isolated singularity,  it is well-known that the local index $\ind_{p} (\grad h )$ equals the  \emph{Minor number} $\mu_{h}$,  which has several topological and algebraic interpretations.

\smallskip

We address here the problem of computing the index at infinity of $f$  by using the topological behaviour of $f$ at infinity, and more precisely its atypical fibres.   
  Let us briefly recall (cf Definition \ref{d:atypical}) that a fibre $f^{-1}(\lambda)$ of a polynomial function $f:\bK^{n} \to \bK$, where $\bK =\bR$ or $\bC$,  is \emph{typical} if  $f$ is a trivial C$^{\ity}$-fibration over some neighbourhood of $\lambda$, and that the \emph{bifurcation locus} of $f$  is the minimal subset $\Atyp f\subset \bK$ such that the restriction $f_{|}:  \bK^{n}\m \Atyp f \to \bK$ is a locally trivial C$^{\ity}$-fibration.   
   The case of  2 variables is the only one where a complete characterisation of $\Atyp f$ is known, e.g. by Suzuki's study \cite{Su}. Motivated  by the Jacobian Conjecture,  \cite{Su} treated the complex setting showing that the variation of the Euler characteristic detects precisely which regular fibre is atypical or not; see also the subsequent contribution \cite{HaLe}, and in more general settings \cite{HNt}, \cite{JoTi2}. The real setting is more delicate, and a complete characterisation in 2 variables occurred more than two decades later \cite{TZ}, see also the subsequent contributions,  e.g. \cite{CP},  \cite{HNl}, \cite{DTT},  \cite{DJT}.


Our study uses a method which permits to approach Arnold's local index formula as well as the more complex computation of the index at infinity, namely the  \emph{Milnor locus of $f$} (Definition \ref{d:milnorser}). This has already been used by several authors in the study of the topology of function germs, probably starting with Milnor's lecture notes \cite{Mi1}, and  in the study of  the change of topology of fibres of polynomials at infinity, cf \cite{NZ},  \cite{Ti-reg}, \cite{DiTi}, \cite{DJT}.
We introduce Milnor arcs at infinity (in Section \ref{s:arcs}),  define the clusters of such arcs,  and show how these clusters detect the bifurcation locus of $f$ via the phenomena of splitting and vanishing at infinity. 
Our first main result, Theorem \ref{t:main1}, tells that $\ind_\ity f$ can be expressed 
in terms of the numbers of vanishing and splitting components of the fibres of $f$.

 The Milnor arcs at some point $p$ on the line at infinity of the projective compactification $\bP^{2} \supset \bR^{2}$ come with an index, which may be $+\frac12$, $0$ or $-\frac12$, and their sum defines the local index at infinity $i_{p}$. Our key Lemma \ref{l:d_{p}} highlights the inequality   
 $i_{p}\le d_{p}-1$
 observed by  Durfee in \cite{Dur2},  where $d_{p}$ is the order at the point $p$ of the top degree part $f_{d}$ of $f$. We establish the origin of  the ``gaps'' which produce the difference between the two sides, we classify these gaps and we explain how to track down their occurrence.
 
  It is not difficult to show that  $| \ind_\ity f | \le d-1$  (by Bezout's theorem), and
 that the lower bound  $1-d$  is realised for instance  by a generic arrangement of lines.\footnote{Notice that this is far from satisfying the Poincar\'e-Hopf theorem, due to the non-compacity of  $\bR^2$.} 
Durfee \cite{Dur2}  proved  the inequality:  
\begin{equation}\label{eq:durfee0}
  \ind_\ity f \le \max\{1,  d-3\}
\end{equation}
and raised the problem of estimating a better upper bound, since many examples show that the index at infinity may be quite far from $d-3$.

  In order to get a grip on sharper upper bounds one has to consider more invariants than just the degree of the polynomial. To this aim, we use the knowledge on the intrinsically defined atypical values and atypical points at infinity.   Theorem \ref{t:upbound}  exploits the classification  provided by Lemma \ref{l:d_{p}} and casts new invariants into a formula which lowers the bound \eqref{eq:durfee0} found by Durfee in  \cite{Dur2}, as follows:
\[
  \ind_{\infty}(f) \le 1 + d_{Re} -  2|\cL_{f}| - \sum_{p\in  L^{\ity}\cap \{f_{d} =0\}} \left(\frac12 \Bigl \lfloor  \frac{\deg(R_{p}^{\red})-1}{2} \Bigr \rfloor + \deg(S_{p}) + \deg(K_{p}) \right),
\]
where $d_{Re}$ is the real degree of $f$,  $\cL_{f}$ is the set of points at infinity of $f$,   $L^{\ity}$ is the line at infinity and  $R_{p}, S_{p}, K_{p}$ are certain subvarieties of the tangent cone to the Milnor set at $p\in L^{\ity}$ (cf \S \ref{upperbd} for all these definitions and notations). Corollary  \ref{c:upbound2} offers two simpler upper bound estimations.

It appears that the case of a single point at infinity is responsible for the indices closest to Durfee's bound  \eqref{eq:durfee0}. In \S\ref{s:durfee} we revisit and complete Durfee's proof of \eqref{eq:durfee0} in \cite{Dur2}, and we clarify a couple of shadow points in it. Our Theorem \ref{t:l=1} is a slight improvement, in particular we show:
\emph{If $f$ has a single point at infinity, then $\ind_{\infty}(f) \le d-3$ for $d \ge 4$, and $\ind_{\infty}(f) \le 0$ for $d\le 3$}.  
  
  Section \ref{s:fibres} concerns properties of fibres of polynomials and their connected components. We review some of the Durfee results \cite{Dur2}  and solve some unclear points in his paper, see for instance Lemma \ref{l:circle}, the remark and the corollary following it.
  
  In  Section \ref{s:arcs} we provide the necessary preparation for the statements and proofs of Theorem \ref{t:main1}, Lemma \ref{l:d_{p}} and Theorem \ref{t:upbound}.    In particular we show the key Propositions
   \ref{p:vantangency} and \ref{l:tangent-nontangent}.    Section \ref{examples} contains examples with explicit computations of the index at infinity and of the ingredients of our main formulas.
  
  \
  
\noindent \emph{Acknowledgement.} Special thanks to the anonymous referee for the pertinent comments and remarks that helped us to improve several proofs.
  

  \section{Fibres of polynomials}\label{s:fibres}

Let $f : \bR^{2} \to\bR$ be a polynomial function of degree $d\ge 2$, and let $f_{d}$ denote its degree $d$ homogeneous part.

\begin{definition}\label{d:atypical}
We say that $\lambda \in \bR$ is a \emph{typical value}, or that $f^{-1}(\lambda)$ is a \emph{typical fibre},  of $f$ if the restriction $f_{|}: f^{-1}(D) \to D$ is a trivial C$^{\ity}$-fibration, for some small enough disk $D\subset \bR$ centred at $\lambda$. We also say that $\lambda \in \bR$ is a \emph{typical value of  $f$ at infinity} if there is a disk $D\subset \bR$ centred at $\lambda$ and a large enough ball $B\subset \bR^{2}$ centred at the origin, such that the  restriction $f_{|}: f^{-1}(D) \cap (\bR^{2}\m B) \to D$ is a trivial fibration. 

 The set  $\Atyp f$   of \emph{atypical values of  $f$}  (also called the \emph{bifurcation set of $f$}) is the minimal subset of $\bR^{2}$ which contains the critical set $f(\Sing f)$ and such that the restriction $f_{|}:  \bR^{2}\m \Atyp f \to \bR$ is a locally trivial C$^{\ity}$-fibration.
\end{definition}

In more than 2 variables there is no complete characterisation of atypical values and the studies focused on finding effective approximations of the bifurcation locus, in particular finding upper bounds for the number of atypical values in terms of the degree (and possibly of some other data), cf \cite{JK},  \cite{Je}, \cite{DiTi}, \cite{JeTi2}. The problem of finding such upper bounds is equally important in 2 variables, see e.g. \cite{LO}, \cite{GP}, \cite{Gw}, \cite{JeTi}. 

 Let $X:= \{ \tilde f(x,y, z) -tz^{d} =0 \} \subset   \bP^{2}\times \bR $, where $\tilde f$ denotes the homogenization of $f$ of degree $d$. Let $\tau : X \to \bR$ be the projection on  the second factor, and let $X_{t}:= \tau^{-1}(t)$ be the fibre of $\tau$ over $t$. Let $L^{\ity} := \{z=0\}\simeq \bP^{1}$ be the line at infinity of  $\bP^{2}$. The part at infinity $X^{\ity}:= X \cap (L^{\ity}\times \bR) = \{ f_{d} =0\} \times \bR$ consists of finitely many lines.
  The algebraic space $X$ may be endowed with a Whitney stratification such that $X^{\ity}$ is a union of strata. This consists of the open stratum $\bR^{2}\subset X$ of dimension 2, and the finitely many strata of $X^{\ity}$ which are either of dimension 1 or of dimension 0. Let us denote by $\cS_{0}$ the finite set which is the union of these strata of dimension 0.    

    Then $X_{t}\subset  \bP^{2}$ is a closed set and contains the closure $\overline{F_{t}}$ of the fibre $F_{t} := f^{-1}(t)$. 
  The part at infinity $X_{t}\cap L^{\ity} =\{ f_{d} =0\}\subset \bP^{1}$ is a finite set and it is independent of $t\in \bR$.
We have the inclusion $\{ f_{d} =0\} \supset \overline{F_{t}}\cap L^{\ity}$, which may be strict, like in the example  $f = x^{2}+y^{4}$ where we have $\overline{F_{t}} \cap L^{\ity} = \emptyset$, but $X_{t} \cap L^{\ity}=[1:0]$ for all $t\in \bR$.
   

 The map $\tau : X \to \bR$ has as singular locus $\Sing \tau = \cS_{0} \cup \Sing f$. In particular the set $\tau(\Sing \tau)$ of critical values of $\tau$ is finite, and  it is well-known, see e.g. \cite{Ti-reg}, \cite{TZ}, \cite{Ti-book}, that we have the inclusion $\Atyp f \subset \tau(\Sing \tau)$.

  The set $\bR\m \Atyp f$ is therefore a union of intervals, two of which are unbounded (and they coincide with $\bR$ in case $\Atyp f = \emptyset$). Let us denote the infinite intervals by $I_{+}$ (the one towards $+\infty$),  and  by $I_{-}$ (the one towards $-\infty$).  Consequently $f_{|} : f^{-1}(I) \to I$ is a trivial fibration, where $I$ is either $I_{+}$ or $I_{-}$.
 
Let then $F_{-}$ and $F_{+}$ denote the fibre of $f$ over some point of $I_{-}$, and of $I_{+}$, respectively.

\begin{lemma}\label{l:circle}
   Assume that the fibre $F_{+}$ contains a compact component $C$. Then $F_{+} =C$ and $F_{-}$ is empty, and 
moreover, all the non-empty fibres of $f$ are compact.

 The same statement holds if we switch the roles of  $F_{-}$  and $F_{+}$.
\end{lemma}

\begin{remark}\label{r:durfeeerror1}
In  \cite[Prop 4.3, point 2 of the conclusion]{Dur2} it is stated that if  all fibres of $f$ are compact or empty, then $f_{d}$ has no linear factors.
  A simple polynomial $f=x^4+y^2$ shows that this conclusion is not true: we have $d=4$, all  fibres are compact or empty,  and $f_4=x^4$ has four linear factors $x$.
\end{remark}

\begin{proof}[Proof of Lemma \ref{l:circle}]
  Let $F_{+} = f^{-1}(a)$ for some $a\in I_{+}$.
 Consider the surface $S_{+} :=f^{-1}\bigl( [a, +\infty)\bigr) \subset \bR^{2}$. 
  
 Let us denote by $C_{+}$ the connected component of $S_{+}$ which contains $C$. The family of ovals $C_{+}$ has been considered and studied in \cite{JoT}, to which we refer the reader for more details.
We claim that $C_{+}$ equals the exterior of the oval $C$. Indeed, along each direction outside the oval $C$, the value of the function $f$ tends to infinity. Therefore the fibres of $f$ over  $[a, +\infty)$ must fill in the region of $\bR^{2}$ outside $C$. Since this region coincides with $S_{+}$ and it is connected,  and since it is the total space of the trivial fibration $f_{|} : S_{+}\to [a, +\infty)$, we deduce the equality  $S_{+}=C_{+}$. The hypothesis that $f$ is a trivial fibration over the interval $I_{+}$ then implies that the fibre $F_{+}$ coincides with the oval $C$. In particular the fibres of $f$ over $I_{+}$ are compact.
 
Moreover, this also shows that $f$ takes values less than $a$ inside the oval $C$, and  since $f$ is bounded inside the compact oval $C$, it follows that the fibres $f^{-1}(t)$ are compact for $t\le a$, and that they are empty for all $t <b$ for some value $b\in \bR$.  
\end{proof}

\begin{definition}\label{d:ell}
 Let $\cL_{f } :=\{\overline{F_{t}} \cap L^{\ity} \mid t\in \bR \}$, which is a finite set of points included in the part at infinity
 $\{f_{d}=0 \}$, and let $|\cL_{f}| := \# \cL_{f }$.
\end{definition} 

By definition $\cL_{f }$ collects all the ``points at infinity'' of the fibres of $f$. Unlike the complex setting where the image of $f_{\bC}$ is  $\bC$ and all fibres have the same points at infinity, in our real setting some fibres may be empty (Example \ref{s:exa4}), some fibres may be compact and some others not (Example  \ref{s:exam2}), or those fibres which contain non-compact components may not have all the same points at infinity (Example \ref{ex:Hlines}).

\begin{corollary}\cite[Prop 4.4(1) and Prop 4.3]{Dur2} \ 
 If $p\in \cL_{f }$ then $p\in \overline{F_{+}}\cap L^{\ity}$ or $p\in \overline{F_{-}}\cap L^{\ity}$.
\end{corollary}
\begin{proof}
 We consider  $D := X\m (\{p\}\times \bR)$, which is a connected set, and we apply the proof of Lemma \ref{l:circle}  to $D$ instead of $\bR^{2}$, and to the fibres $X_{t}$ restricted to $D$. 
 Namely, by contradiction,  if $p\not\in \overline{F_{+}}\cap L^{\ity}$ then $\overline{F_{+}}$ is compact in $D$ and so, by  Lemma \ref{l:circle},  all the fibres $X_{t}\cap D$ are compact in $D$ or empty. This shows that $p\not\in X_{t}$ for any $t\in \bR$, which contradicts our hypothesis.
\end{proof}

We give an account of  Durfee's proof  of the following result, for the reader's convenience. 
\begin{proposition}\label{p:comp-at-inf}\cite[Cor 4.2 and Prop. 4.4(2)]{Dur2}
 The number of  connected components of $F_{+}\cup F_{-}$ is at least $2|\cL_{f}|$.
\end{proposition}
 	
\begin{proof}
For the reader's convenience, we recall here  Durfee's proof.  One considers a sequence of blow-ups at points at infinity in $\bP^{2}$ which yields a \emph{resolution at infinity of $f$}. This produces a space $M$ and a proper map $\hat f : M\to \bR$ which extends $f$. One may regard the space $M$ as the disjoint union of $\bR^{2}$ and a finite number of divisors at infinity. Some of these divisors (denoted by $E$) are ``horizontal'', i.e. the restriction $\hat f_{|E} : E\to \bR$ is a non-constant polynomial of one variable, and the others are ``vertical'', i.e. the restriction $\hat f_{|E} : E\to \bR$ is constant. 

For each point of $\cL_{f}$ there is at least one horizontal divisor. Considering a horizontal divisor $E$, 
    if $\hat f_{|E}$ is a polynomial of odd degree, then it is injective outside a compact interval $[-A, A]$, and therefore $(\hat f)^{-1}(a)$ has one solution for every $a$ such that $|a| > A >0$.
If the degree is even, then we have two solutions towards  $-\ity$ and no solution towards $+\ity$, or the other way around.
To each such solution corresponds a local branch of $(\hat f)^{-1}(a)$, and for each half-branch we count one connected component of  the fibre $f^{-1}(a)$.  
In this way each connected component of $f^{-1}(a)$ is counted twice, whether or not its two intersection points 
with the horizontal divisors  coincide.

We therefore obtain that $\# F_{+}\cup F_{-}$ equals the double of the number of the horizontal divisors. In particular we get the claimed inequality $\# F_{+}\cup F_{-} \ge 2|\cL_{f}|$. 
 \end{proof}
 
  The following example by Durfee shows that the inequality of Proposition  \ref{p:comp-at-inf} may be very far from an equality.
\begin{example} \cite[p. 1347]{Dur2} \label{ex:Hlines}
 Let $f=x(y+1) \cdots (y+k)$, $k\ge 2$.  The fibre $f^{-1}(0)$ produces a partition of the plane into $2(k+1)$ horizontal strips between parallel lines delimited by the vertical axis. Each horizontal strip contains a connected component of $F_{+}\cup F_{-}$, whereas $|\cL_{f}|=2$.
\end{example}

 
 \section{Milnor arcs and clusters at infinity}\label{s:arcs}

 \subsection{Milnor arcs at infinity}\label{ss:arcs}  

The idea of the Milnor set was introduced by Milnor \cite{Mi1} and has been used in many papers ever since, either locally for germs of functions \cite{Th, AFS,Sa, Be,ACT1,AT}, or globally for polynomial functions \cite{NZ, Ti-reg,Ti-book, DRT,ACT2,   DJT, Mo}.

Let $\rho_a : \bR^2 \rightarrow \bR _{\geq0}  $, $\rho_a(x,y)=(x-a_1)^2+(y-a_2)^2 $ be the square of the  Euclidean distance to $a:= (a_1 , a_2) \in \bR^2$. 
\begin{definition}\label{d:milnorser}
The Milnor set of  $f : \bR ^2 \rightarrow \bR$ relative to $\rho_{a}$ is the set of $\rho_a$-\textit{nonregular} points of $f$, namely
$$M_a(f) := \{ (x,y) \in \bR^2 \mid  \rho_a \not\pitchfork_{(x,y)} f  \}.$$
\end{definition}
Equivalently, the \textit{Milnor set} $M_a(f)$ is the zero set $\{\Jac F_{a} =0\}$ considered with reduced structure, where $\Jac F_{a}$ denotes the determinant of the Jacobian  matrix of the map $F_a:= ( f, \rho_a ): \bR^2 \rightarrow \bR^2$.

\begin{propdef}[Milnor arcs at infinity] \cite{Mo}\label{p:milnorarcs}\\
Let $f : \bR^2 \to\bR$ be a non-constant polynomial. There exists a dense subset $\Omega(f)\subset \bR^{2}$ of points such that  $M_{a}(f)$ is of dimension 1 for any $a\in \Omega(f)$. For each such point $a \in \Omega(f)$,  there exists a radius $R_{a}\gg 1$ such that for any $R\ge R_{a}$, and denoting by $D_{R}(a)\subset \bR^{2}$ the closed disk centred at $a$ of radius $R$, one has:

\begin{enumerate}
\rm \item \it  The set $M_{a}(f) \m \left[ D_{R}(a) \cup \Sing f \right]$ is a disjoint union of finitely many 1-dimensional connected manifolds, that we denote by $\gamma_{1}, \ldots, \gamma_{s}$.

\rm \item \it One may endow each $\gamma_{i}$ with a parametrisation $\gamma_{i} : ]R, +\infty[ \to \bR^2$ such that
the restriction $(\rho_{a})_{|\gamma_{i}}$ is strictly monotonous and tends to infinity as the parameter $t$ tends to infinity; we call \emph{Milnor arc at infinity} the parametrised curve $\gamma_{i}$.
\rm \item \it for every Milnor arc at infinity, the restriction $f_{|\gamma_{i}(t)}$ is either:

$\bullet$ strictly increasing as $t\to +\ity$, and if $\lim_{t \to \infty} f(\gamma(t)) = \lambda\in \bR\cup \{+\ity\}$, then we say that $\gamma$ is an \emph{increasing Milnor arc at infinity of $f$ associated to $\lambda$}, and abbreviate this by $f_{|\gamma }{\nearrow} \lambda$,

\noindent or 

$\bullet$ strictly decreasing as $t\to +\ity$, and if $\lim_{t \to \infty} f(\gamma(t)) = \lambda\in \bR\cup \{-\ity\}$, then we say that $\gamma$ is a \emph{decreasing Milnor arc at infinity of $f$ associated to $\lambda$}, and abbreviate this by $f_{|\gamma }{\searrow} \lambda$.     
\end{enumerate}
 \fin
\end{propdef}

\begin{definition}
Any Milnor arc $\gamma$  has a unique point at infinity $p\in L^{\ity}\cap \overline{\gamma}$; we shall say that ``the Milnor arc $\gamma$ has the point $p$ at infinity''.
\end{definition}

We may and will  assume  from now on that, modulo a translation of coordinates, the origin 0 is a point of $\Omega(f)$, and for this point we will use the simplified notations without lower index, such as $M(f)$  etc. 

\begin{remark}\label{r:tz}
  Unlike the setting of complex polynomials of 2 variables where the existence of the Milnor set at a point at infinity is a precise indicator of  the existence of an atypical fibre (see e.g. \cite{ST}, \cite{Ti-book}), in the real setting this is no more true.  For instance in \cite[Example 3.2]{TZ} $0\not\in \Atyp f$ but there are Milnor arcs at the point $[1:0:0] \in L^\ity$ and $f$ tends to the value 0 along each of these arcs. 
\end{remark}

\subsection{Clusters of Milnor arcs}\label{ss:clusters}

By Proposition \ref{p:milnorarcs}, the Milnor arcs at infinity do not intersect mutually. It follows that if $C\subset \bR^{2}$ is some large enough circle centred at the origin, then $M(f) \cap C$ is a finite set of points $\{p_{1}, \ldots, p_{s}\}$.  We define the following counterclockwise relation between these points\footnote{Note that this is not an order relation.}: we say that ``\emph{$p_{j}$ is the successor of $p_{k}$}'', or that ``\emph{$p_{k}$ is the antecedent of $p_{j}$}'',   if  starting from the point $p_{k}$ and moving counterclockwise along the circle $C$ one arrives at the point $p_{j}$  without meeting any other point of the set $M(f) \cap C$.  

We also say that $\{p_{1}, \ldots, p_{k}\}$ is a sequence of consecutive  points of the set $M(f) \cap C$ if $p_{i+1}$ is the successor of $p_{i}$ for all $i=1, \ldots, k-1$.   This relation between the points $M(f) \cap C$ on the circle $C$ allows us to define a similar one among the Milnor arcs at infinity, as follows:  

\begin{definition}[Counterclockwise ordering of Milnor arcs at infinity]\label{d:orderarcs} \ \\
We say that
``\emph{$\gamma_{j}$ is the successor of $\gamma_{k}$}'', or that ``\emph{$\gamma_{k}$ is the antecedent of $\gamma_{j}$}'',  if  the point $p_{j} := \gamma_{j}\cap C$ is the successor of the point $p_{k} := \gamma_{k}\cap C$. This relation is independent on the size of the circle $C$, provided large enough.  We also say that $\{\gamma_{1}, \ldots, \gamma_{k}\}$ is a sequence of consecutive  Milnor arcs at infinity if $\{p_{1}, \ldots, p_{k}\}$, where $p_{i}:= \gamma_{i}\cap C$, is a sequence of consecutive  points of the set $M(f) \cap C$.
\end{definition}

\begin{definition}[Clusters of Milnor arcs at infinity]\label{d:cluster}
We call \textit{increasing cluster at $\lambda\in \bR  \cup \{+\infty\}$} a sequence of consecutive Milnor arcs at infinity $\gamma_k, \ldots ,  \gamma_{k+l}$, $l\ge 0$,  such that the condition $f_{ | \gamma_i} \nearrow \lambda$ holds precisely for all $i= k, \ldots , k+l$ and does not hold for the antecedent of $\gamma_k$ nor for the successor of $\gamma_{k+l}$.

Similarly, we define a \textit{decreasing cluster at} $\lambda  \in \bR  \cup \{-\infty\}$ by replacing $\searrow $ instead of $ \nearrow$ in the above definition. We will use the generic name ``Milnor cluster'', or simply ``cluster'', for any increasing or decreasing cluster.
\end{definition}
A similar definition of Milnor clusters was given in \cite{HNl} in  the  setting of surfaces in 
$\bR^{n}$ instead of $\bR^{2}$. Earlier, polar clusters  have been defined in \cite{CP}. In \cite{CP} and  \cite{HNl},  clusters are used for detecting atypical fibres. An effective detection of atypical values via Milnor clusters can be found in \cite{Mo}.  Let us point out that \cite{DJT} develops an algorithmic detection of atypical fibres without using Milnor clusters.

\begin{theorem}[\cite{HNl}, \cite{Mo}]\label{t:odd-cluster}
 Let $f:\bR^{2}\to \bR$ be  a non-constant polynomial function, and let $\lambda\in \bR$ such that the fibre $f^{-1}(\lambda)$
 has at most isolated singularities. Then $\lambda$  is an atypical value of $f$ at infinity\footnote{See Definition \ref{d:atypical}.} if and only if there exists a cluster at $\lambda$ (either increasing or decreasing) having an odd number of  arcs. 
 \fin
\end{theorem}

\begin{remark}\label{r:component}
 It was shown in \cite{Mo} that for any cluster $\cC$ at $\lambda\in \bR  \cup \{\pm\infty\}$ there is a unique connected component of the fibre $f^{-1}(t)\m D_{R}$, denoted by $\alpha_{t}(\cC)$,  which intersects all the Milnor arcs of the cluster $\cC$, for  $t$ close enough to $\lambda$. It was proved in  \cite{Mo} that the correspondence  $\cC\mapsto \alpha_{t}(\cC)$, for  a large enough disk $D_{R}$, is a well-defined map from the set of clusters to the set of fibre components in $\bR^{2}\setminus D_{R}$, which is moreover injective.

Let us also point out that two different components $\alpha_{t}(\cC)$ of $f^{-1}(t)\m D_{R}$ may belong to the same connected component of the
 affine fibre $f^{-1}(t)$, and  see Example \ref{s:exam2} for such a situation. 
\end{remark}

\smallskip

\noindent \textbf{Convention}. In the rest of this paper we shall designate the Milnor arcs at infinity simply as ``Milnor arcs''.

 \subsection{The Splitting (Sp) and the Vanishing (Va) at infinity}\label{ss:sp-van}\ \\
  The splitting (denoted Sp) and the vanishing (denoted Va) of fibre components are phenomena which may happen,  at some point $(p, \lambda)\in \cL_{f}\times \bR$ where $\lambda$ denotes a value of $f$. These have been defined\footnote{For related viewpoints and for extensions one may consult \cite{CP}, \cite{HNl}, \cite{DTT}, \cite{DJT}.} in \cite{TZ}. They are related to Milnor arcs at infinity. More precisely, we will see in the following that  any odd cluster is either a \emph{splitting cluster} or a \emph{vanishing cluster}. 
 
After \cite{DJT} (see Theorem \ref{t:djt} below), the existence of atypical fibres is equivalent to the existence of  \textit{atypical points at infinity} in  $\cL_f \times \bR$, which are defined with respect to the local splitting and local vanishing (which are the localisations of the Sp and Va phenomena).

 In order to display the definitions, let us recall a few preliminaries following \cite{DJT}.  
Let $ \{ M_t \}_{t\in \bR}$ be a family of sets in $ \bR^2$. We say  that the \textit{limit set} of the family $ \{ M_t \}_{t\in \bR}$ when $ t \to \lambda $, and we denote it by $ \lim _{t\to \lambda} M_t$,  is the set of points $ x \in \bR^2$ such that there exists a sequence $t_k \in \bR $ with  $ t_k \to  \lambda$ and a sequence of points $ x_{k} \in M_{t_k}$ such that $ x_{k} \to x$.  
  
\begin{definition}\label{d:van-spl}\cite{DJT} 
Let $ \lambda \in \bR$ such that $ \text{Sing} f^{-1}(\lambda)$ is a compact set. 

\sloppy
\noindent 
(i) One says that $ f$ \textit{has a vanishing at infinity at $ \lambda$}, 
 if $\lim_{t \to \lambda^{-}} \mathop{\text{max}_j} \mathop{ \text{inf}_{q \in F_{t,j}}} \| q\|= \infty$, or $\lim_{t \to \lambda^{+}} \mathop{\text{max}_j} \mathop{ \text{inf}_{q \in F_{t,j}}} \| q\|= \infty$, where $j$ runs over all connected components  $F_{t,j}$ of the fibre $f^{-1}(t)$.
 

\medskip
\noindent
(ii) One says that $f$ \textit{has a splitting at infinity at $\lambda$},   if there exists $ \eta >0$ and a continuous family of analytic paths $ \phi_{t} : \left[ 0,1  \right] \to f^{-1}(t)$ for $t \in \left(  \lambda - \eta , \lambda \right)$, or for $t \in \left( \lambda, \lambda + \eta \right)$, such that:

\smallskip

\noindent
(1) $\im\phi_t  \cap M(f) \neq \emptyset$, and $ \lim _{t \to \lambda}\|q_t\| = \infty$    for any $q_t \in \im\phi_t \cap M(f)$, 

and 

\noindent
(2) the limit set $\lim_{t \to \lambda^{-}} \im \phi_t$, or $\lim_{t \to \lambda^{+}} \im\phi_t$, respectively,  is not connected.
\end{definition}
\smallskip

\begin{definition}\label{d:oddclusters}
 We say that a cluster  $\cC$ is \emph{odd} (or \emph{even}) if it has an odd number of Milnor arcs (or an even number of Milnor arcs, respectively). 

By  Remark \ref{r:component},  there is an injective correspondence between the clusters $\cC$ at $\lambda \in \bR  \cup \{\pm\infty\}$ and the connected components of the fibre $f^{-1}(t) \m D_{R}$ for $t\to \lambda$ and large enough radius $R$.  We have denoted by 
$\alpha_t(\cC)$ the connected component defined by $\cC$ in this correspondence. It was proved in 
\cite[Proof of Theorem 6.5]{Mo} that $\cC$ is an odd cluster if and only if  $\alpha_t(\cC)$ is either vanishing or splitting at infinity, in the sense of Definition \ref{d:van-spl}(i-ii), where one replaces $f^{-1}(t)$ by $\alpha_t(\cC)$.

In the following we will therefore call such an odd cluster $\cC$ at $\lambda$  either a \emph{vanishing cluster}, or a  \emph{splitting cluster}, accordingly. 
\end{definition}

The paper \cite{DJT} shows that one can \emph{localise} at some points $(p,\lambda) \in L^{\ity} \times \bR$ the vanishing and the splitting at infinity of $f$ at $\lambda$ (cf Definition \ref{d:van-spl}). To explain this result that we will need here, let us recall some notations and definitions.

By a linear change of coordinates we may assume, without loss of generality, that $p \in L^\ity$ is the point $[0:1:0]$. Recall that $\tilde f (x,y,z)$  denotes the homogenization of degree  $d= \deg f$ with respect to the new variable $z$.    In some chart  $U \simeq \bR^2 \subset \bP^2$ at $p$, the family of polynomial functions $g_t:= \tilde f(x,1,z)-tz^d $ defines a family of algebraic curve germs $C_t:=\{g_t =0\}$ at $p$, of parameter $t$.  
\begin{definition}\label{d:vanloop-split} \cite{DJT} (Localisation).

 One says that $f$ has a \emph{splitting} at $(p,\lambda) \in L^\ity \times \bR$ if there is a small disk $D_\e$ at $p$ in some chart at infinity $\bR^2$ such that the representative of the curve $C_t$ in $D_\e$ has a connected component $C^i_t$ such that $C_t^i \cap \partial D_\e \neq \emptyset$ for all $t > \lambda$ (or for all $t<\lambda$) close enough to $\lambda$, and that the local Euclidean distance $\dist(C_t^i,p)\neq 0$ tends to $0$ when $t \to \lambda$.

One says that $f$ has a \emph{vanishing} at $(p,\lambda) \in L^\ity \times \bR$  if there is a small disk $D_\e$ at $p \in U$ such that $C_t \cap D_\e\m \{p\}$ has a non-empty connected component $C_t^i\m \{p\}$ with $C_t^i \cap \partial D_\e = \emptyset$ for all $t<\lambda$ (or for all $t>\lambda$) close enough to $\lambda$, such that $\lim_{t \to \lambda}C_t^i\cap D_\e=\{p\}$.  
\end{definition}

\begin{definition}\label{d:atypnt} \cite{DJT} (Atypical points at infinity).

We say that $(p,\lambda)\in L^\ity \times \bR$ is an \textit{atypical point at infinity} of $f$ if there is either splitting or vanishing at $(p,\lambda)$.
\end{definition}


%
\begin{theorem}\label{t:djt} 
A  value $\lambda\in \bR$ is an atypical value at infinity of $f$ (cf Definition \ref{d:atypical}) if and only if there exists $p\in \cL_{f}$ such that $(p,\lambda)$ is an atypical point at infinity. 

More precisely, if $\cC$ is a splitting cluster at $\lambda$ (cf Definition \ref{d:oddclusters}), then after splitting, the two local fibre components have the same point $p$ at infinity, and if  $\cC$ is a vanishing cluster at $\lambda$, then before vanishing, the fibre component has the unique point $p$ at infinity.  
\end{theorem}
\begin{proof}
 The first claim was proved in \cite[Theorem 1.1, Theorem 2.10]{DJT}. 
 To show the second claim,  
 let $\alpha_t(\cC)$ be the unique connected component of $f^{-1}(t) \m D_R$ (for a radius $R$ large enough, and for $t$ close enough to $\lambda$)  which corresponds to the cluster $\cC$ by  Remark \ref{r:component}.
 
 If $\cC$ is a vanishing cluster at $\lambda$ then for $t$ close enough to $\lambda$,  $\alpha_t(\cC)$ is a loop at some point $p\in L^\ity$, and therefore $(p, \lambda) \in L^\infty \times \bR$ is the unique vanishing point of $\alpha_t(\cC)$ for $t\to \lambda$.  Since all the Milnor arc in the cluster $\cC$ intersect $\alpha_t(\cC)$, it follows that all of them have the point $p$ at infinity. 
 
 If $\cC$ is a splitting cluster at $\lambda$ then, as $t\to \lambda$,  $\alpha_t(\cC)$ splits at least at some 
 point $p\in L^\ity \cap \overline{\alpha_t(\cC)}\subset L^\ity \cap \overline{f^{-1}(\lambda)}$. The limit set $Z :=\lim_{t \to \lambda} \alpha_t(\cC)$ is then non-connected (see \S \ref{ss:sp-van} for the definition of the limit set).
  By contradiction, if $\alpha_t(\cC)$ splits at more than one point, then the limit set $Z\subset f^{-1}(\lambda)\m D_{R}$ contains at least an arc $A \simeq \bR$ which has at infinity two such points. Clearly, the  arc $A$ cannot be contained in the exterior of the disk $D_{R}$ for any large enough $R$.  As $A$ is part of the limit set $Z$, it then follows that the nearby fibre component $\alpha_t(\cC)$ has the same property. This means that $\alpha_t(\cC)$ has at least two connected components in $\bR^{2}\m D_{R}$ for some large enough $R$, which is a  contradiction to its definition. 
The proof of the unicity of the splitting point $p$ is now complete.
  
  Finally, since all the Milnor arcs in the cluster $\cC$ intersect $\alpha_t(\cC)$, it follows that all of them have this point $p$ at infinity.
\end{proof}

By comparing the proof of Theorem \ref{t:djt} to Definition \ref{d:atypnt}, we immediately get the following rephrasing\footnote{One has a similar result for even clusters, with a similar type of proof.}:

\begin{corollary}\label{c:point}
To an odd cluster $\cC$ at $\lambda\in \bR$ there corresponds a unique atypical point $(p,\lambda)\in \cL_f \times \bR$, such that all the Milnor arcs in the cluster $\cC$ have the same point $p$ at infinity.
\fin
\end{corollary}
  However,  let us note that Corollary \ref{c:point} is not anymore true for odd clusters corresponding to connected components of fibres of $f$ which tend to the values $\pm \ity$, see Example \ref{s:exa4}.


\subsection{Points at infinity, clusters and tangents}\label{ss:gen} \

 We recall that the notation $\cL_{f}$ stands for the points at infinity of all the fibres 
of a non-constant polynomial $f: \bR^{2}\to \bR$, and $|\cL_{f}| = \# \cL_{f}$. 
Also recall that a vanishing or a splitting cluster at some value $\lambda$ contains an odd number of Milnor arcs, by Theorem \ref{t:odd-cluster}, and that all these arcs contain the point $p$ in their closure at infinity, by Corollary \ref{c:point}. 

\begin{proposition}\label{p:tangency}\label{p:vantangency}
Let $\cC$ be an odd cluster and let $(p, \lambda)\in \cL_{f}\times \bR$ be its unique atypical point at infinity (cf Corollary \ref{c:point}). Then:
\begin{enumerate}
\rm \item \it If $\cC$ is a splitting  cluster then,  after splitting, the resulting  two germs at $p$ of fibre components,  denoted by $C_{p,1}$ and $C_{p,2}$,  have a common tangent semi-line, call it  $T$, and all the Milnor arcs of the cluster $\cC$ are also tangent to the same semi-line $T$ at $p$.
\rm \item \it  If $\cC$ is a vanishing cluster, let $C_{1}$ and $C_{2}$ be the two local arcs at $p$ of the 
component $\alpha_{t}(\cC)$ which vanishes at $p$ when  $t\to \lambda$.  Then $C_{1}$ and $C_{2}$ are tangent at $p$ to the same semi-line $T$,  and all the Milnor arcs of this cluster are tangent to the same semi-line $T$ at $p$.
\end{enumerate}
\end{proposition}

\begin{proof}
We give the proof for (a) only, as the one for (b) is analogous.

 By contradiction, suppose that $T_{p}C_{p,1}=R$ and $T_{p}C_{p,2}=T$,  for  two semi-lines $R  \not= T$ at $p$.   
Let us then consider some semi-line $L$ at $p$ in the interior of the angle $\delta$ of measure less than $\pi$ at $p$ spanned by the semi-lines $T$ and $R$. 
   For any $t$ close enough to $\lambda$, the component $\alpha_{t}(\cC)$ must intersect $L$: if not, then $\alpha_t(\cC)$, is contained in one of the two angles spanned by $L^\ity$ and the semi-line $L$, which contradicts our above assumption about two different tangent semi-lines $T$ and $R$.

 Now, since $\alpha_{t}(\cC)$ intersects $L$ for any $t$ close enough to $\lambda$, then it follows
  that the restriction $f_{| L}$ of $f$ to the line $L$ is not constant. Therefore $f_{| L}$ must be unbounded, since the restriction $f_{| L}$ is a non-constant polynomial of one variable.
    More precisely, for any $t$ close to $\lambda$, there is a point of intersection $q(t) \in \alpha_{t}(\cC)\cap L$ which   tends to $p$ when $t\to \lambda$, and  thus the value $f_{| L}(q(t))$ must converge to infinity as $t\to \lambda$.  On the other hand, we have 
  $f_{| L}(q(t)) = t$ and the limit $\lim_{t\to \lambda}f_{| L}(q(t))$ is $\lambda$ by construction. This yields a contradiction.
        
To show the tangency to the semi-line $T$ of the Milnor arcs of the cluster $\cC$, let us remark that $\alpha_t(\cC)$, for $t$ close enough to $\lambda$,  is included in the thin region $\cA$ spanned by the splitting components $C_{p,1}$ and $C_{p,2}$ with common tangent $T$. Any Milnor arc in $\cC$ intersects $\alpha_t(\cC)$ for $t\to \lambda$, so it is constraint  by $\cA$ to have the same tangent $T$ at $p$. 
\end{proof}



 \begin{proposition}\label{l:tangent-nontangent}
  Let  $\gamma$ and  $\delta$ be two consecutive Milnor arcs (in the order of arcs, cf Definition \ref{d:orderarcs}) 
    such that they have either different points at infinity, or the same point at infinity $p\in \cL_{f}$ but are tangent to different semi-lines at $p$.  Let 
  $\cC_{\gamma}$ and $\cC_{\delta}$ be their respective clusters.  Then the corresponding fibre components $\alpha_{t}(\cC_{\gamma})$ and  $\alpha_{t}(\cC_{\delta})$ cannot both split.
\end{proposition}

\begin{proof}
First of all, the hypotheses imply, via Corollary \ref{c:point} and Proposition \ref {p:tangency},  that $\cC_{\gamma} \not= \cC_{\delta}$. Suppose then, by contradiction, that both components $\alpha_{t}(\cC_{\gamma})$ and  $\alpha_{t}(\cC_{\delta})$ split at $p$.
We first assume that $\gamma$ and $\delta$ have the same point $p\in \cL_{f}$ at infinity but different tangent semi-lines at $p$.
 The splitting can happen only at atypical values of $f$, so let $\lambda_{\gamma}, \lambda_{\delta} \in \bR$ be the atypical values where $\alpha_{t}(\cC_{\gamma})$ and  $\alpha_{t}(\cC_{\delta})$ split at $p$, respectively.
  
  Let $T$ be the semi-line tangent to $\delta$ at $p$, and let $L$ be the semi-line 
tangent to $\gamma$ at $p$. By our hypothesis, $L\not= T$. 

The component $\alpha_{t}(\cC_{\gamma})$ splits as $t\to \lambda_{\gamma}$ into two branches $C^{1}_{\gamma}$ and $C^{2}_{\gamma}$, and the component $\alpha_{t}(\cC_{\delta})$ splits as $t\to \lambda_{\delta}$ into two branches $C^{1}_{\delta}$ and $C^{2}_{\delta}$. Since there is no other Milnor arc between $\gamma$ and  $\delta$ in the counterclockwise ordering, there must be a family of fibre components between $C^{i}_{\gamma}$ and $C^{j}_{\delta}$, for appropriate $i, j\in \{1,2\}$,
which is a topologically trivial family at infinity (in the sense employed in  Definition \ref{d:atypical}). But 
there cannot be a trivial fibration at infinity since all the fibres in such a trivial fibration must have the same tangent semi-line at $p$. This implies that  there exist an atypical point 
at infinity $(p, \lambda)$, with $\lambda$ in the open interval between  $\lambda_{\gamma}$ and $\lambda_{\delta}$. In turn, this implies that there exists a Milnor arc ``between''
  $\gamma$ and  $\delta$ in the counterclockwise ordering, which is a contradiction to our assumption.  

\smallskip

Let us now assume that the consecutive  Milnor arcs $\gamma$ and $\delta$  do not have the same  point at infinity, and that
the corresponding clusters  are splitting like described above. Then, as observed in the preceding paragraph,  the region $\mathcal{R}_{ij}$ outside a large enough disk $D$ and  between the two corresponding  fibre components  $C^{i}_{\gamma}$ and $C^{j}_{\delta}$ is either filled with a trivial 
fibration (defined by the appropriate restriction of $f$),  or there is no such trivial fibration containing $C^{i}_{\gamma}$ and $C^{j}_{\delta}$. Since the connected components $C^{i}_{\gamma}$ and $C^{j}_{\delta}$ have different points at infinity, they cannot live in a trivial family of connected fibres. But if there is no fibration between $C^{i}_{\gamma}$ and $C^{j}_{\delta}$, then there exists an atypical fibre at infinity in that region $\mathcal{R}_{ij}$, and thus there exists another Milnor arc ``between'' $\gamma$ and $\delta$, which is a contradiction to our assumption that $\gamma$ and $\delta$ are consecutive Milnor arcs.
  \end{proof} 
 
  \begin{remark}
Example \ref{exa1} shows two consecutive Milnor arcs $\gamma$ and  $\delta$ belonging to different clusters,  such that both fibre components $\alpha_t(\cC_\gamma)$ and $ \alpha_t(\cC_{\beta})$ split. However they have the same tangent line. This shows that the hypotheses of Proposition \ref {l:tangent-nontangent} are sharp.
 \end{remark}

\section{Index at infinity via Milnor arcs and clusters} \label{s:index}

Let $f:\bR^{2}\to \bR$ be a non-constant polynomial function with isolated singularities.  We assume as in \S\ref{ss:arcs} that  the origin $0\in \bR^{2}$ is a point in $\Omega(f)$. Let $D$ be a disk centred at the origin of large enough radius such that it contains $\Sing f$ in its interior and satisfies Proposition \ref{p:milnorarcs}.  Let $S^1$ be the unitary circle in $\bR^2$.  The restriction of the \textit{Gauss map} $\psi:=\frac{\grad f}{\|\grad f\|}$  to the circle $C:= \partial D$ defines a $C^\ity$  oriented map $ \psi _{|C}: C \to S^1$ between the circles $C$ and $S^1$ endowed with their counterclockwise orientation.

Durfee introduced in \cite{Dur2}  the \textit{index at infinity of} $f$: 
\begin{equation}\label{eq:indinfty}
 \ind _\ity (f) := \deg (\psi_{|C}).
\end{equation}


\subsection{Index of a Milnor arc, after \cite{Dur2}}\label{ss:ind}
 Recall that the Milnor set $M(f)$ is the set of points where the fibres of $f$ are tangent to the level sets of the Euclidean distance function $\rho$,  and that, by Definition \ref{p:milnorarcs}, the Milnor arcs do not intersect $\Sing f$.  For any point $q$ of a Milnor arc $\gamma$ outside a disk $D  = \{\rho(x,y)  \le R \}$ of large enough radius $R$,  the fibre of $f$ passing through $q$ may be in only one of the following three situations: 
 
\begin{enumerate}
\item locally inside the disk $D$, and then one defines the index $i(\gamma) :=  +\frac{1}{2}$
\item  locally outside $D$, and then one defines the index $i(\gamma) :=  -\frac{1}{2}$,
\item a local half-branch inside $D$ and the other local half-branch outside $D$, in which case one defines the index $i(\gamma) :=  0$. Actually, we will see in the proof of Lemma \ref{l:durfee1} that, for a generic choice of the origin, the Milnor set $M(f)$ does not contain Milnor arcs $\gamma$ of index $0$.
\end{enumerate}

It then follows that along a fixed fibre component, outside the disk $D$,  the distance function has alternating local maxima and minima in the counterclockwise order, and without counting the inflexion points.
Therefore we get:
\begin{lemma}\label{l:consec-index}
The consecutive Milnor arcs (cf. Definition \ref{d:orderarcs}) in the same cluster,  without counting the index 0 arcs, must have alternating index signs.
\fin
\end{lemma}
Now by  Theorem \ref{t:odd-cluster}, Remark \ref{r:component}, and Lemma \ref{l:consec-index}, we directly get the following consequence:
\begin{corollary} \label{c:indexcluster}
Any splitting cluster at $\lambda\in \bR$ has  total index $+\frac12$. Any vanishing cluster at $\lambda\in \bR  \cup \{\pm\infty\}$ has total index $-\frac12$. Any even cluster has total index $0$.
 \fin
\end{corollary}

Let us set the following notation:
$$i_{p,c} := \sum_{\gamma} i(\gamma)$$
 where $\gamma$ runs over all  Milnor arcs  $\gamma$  such that $p\in \overline{\gamma}$ and that $\lim f_{|\gamma}= c$.

   
\begin{lemma}\cite[p. 1356]{Dur2}\label{l:durfee1}
\begin{equation}\label{eq:indexequation}
\ind_{\infty}(f)  = 1 + \sum_{p\in L^{\ity}, \  c\in \bR \cup \{\pm \infty\}} i_{p,c}.
\end{equation} 
\end{lemma}

\begin{proof}
From Definition \ref{d:milnorser}, after identifying $\bR^{2}$ with $\bC$, we get the defining equality:
\begin{equation}\label{eq:milnorequality}
 M(f) =  \left\{ q\in \bC \  \Big\vert \  \frac{\grad f(q)}{\|\grad f(q) \|} = \pm \frac{q}{\|q\|} \right\} .
\end{equation}
Let us therefore consider the oriented $C^\ity$-map $\phi:= \psi _{|C} \cdot ( \frac{z}{\|z\|})^{-1}: C \to S^1$, where $z$ denotes the variable in $\bC$, and where both circles $C$ and $S^1$ are endowed with their counterclockwise orientation. The map $\phi$ is by definition the multiplication of  $\psi _{|C}$ with the clockwise rotation  $(\frac{z}{\|z\|})^{-1} =\frac{\overline{z}}{\|z\|} : C \to S^{1}$ of  degree $-1$.  Using the winding number interpretation of the degree,  one obtains the equality:
\begin{equation}\label{eq:eqind}
  \deg (\phi) = \deg (\psi_{|C}) -1 = \ind_ \ity (f)-1.
\end{equation}

Without loss of generality, we may assume that $1,-1 \in S^1$ are regular values of $\phi$.  We obtain:
\begin{equation}\label{eq:inddef}
 \deg (\phi) =  \frac{1}{2} \left( \sum_{q \in \phi^{-1}(1)}  \ori (T_q  \phi) + \sum_{q \in \phi^{-1}(-1)}  \ori (T_q  \phi) \right)
=  \sum_{q \in M(f) \cap C} \frac{1}{2} \ori (T_q  \phi),
\end{equation}
where $\ori (T_q \phi)$ denotes the orientation of the tangent map, and where the last equality follows  since we have $\phi^{-1}(\{-1, 1\}) =  M(f)\cap C$ in view of \eqref{eq:milnorequality}. 

\smallskip

Let us explain here what are the local orientations $\ori (T_q  \phi)$.
Let $q \in \gamma \cap C$, for some Milnor arc $\gamma$. Referring to the definition in the beginning of \S \ref{ss:ind}, we have the following correspondences (where ``increasing'' means here counterclockwise, and ``decreasing'' means clockwise):

\noindent $\bullet$ In the case (a) the Gauss map  $\psi$ is increasing relative to the radial map $\frac{z}{\|z\|}$,  and thus $\ori (T_q\phi) = +1$. 

\noindent $\bullet$ In the case  (b) the Gauss map  $\psi$ is decreasing relative to the radial map $\frac{z}{\|z\|}$,  and therefore  $\ori (T_q\phi) = -1$. 

\noindent $\bullet$ The case (c) at the point $q$ means that this point is a local maximum or a local minimum for  the map $\phi$, thus a critical point. This situation cannot occur because we have assumed that $q$ is a regular point of $\phi$.

  From \eqref{eq:eqind} and  \eqref{eq:inddef} we then obtain:
 \[ \ind_\ity (f) = 1 + \sum_{\gamma} i(\gamma),\] 
 where the sum runs over all the Milnor arcs of $f$.  
\end{proof}


  In  case $F_{+}\cup F_{-}$ contains  a compact component,  by Lemma \ref{l:circle} we get that $|\cL_{f}| =0$,  and all fibres of $f$ are either compact and connected, or empty. A non-empty fibre of $f$ is then homomorphic to a circle. The winding number of $\grad f$ over such a circle is 1, and it follows that $\ind_{\ity}(f)=1$.
Therefore, in the following  we will tacitly consider only polynomials which have at least a
 non-compact fibre.

Let $\Sp(p, \lambda)$ and $\Va(p, \lambda)$ denote the numbers of connected components of fibres of $f$ outside the large disk $D$  which are splitting or  vanishing,  at the point $(p, \lambda)$, respectively.

Let  $\Va(\pm \infty)$ denote the number of components of $F_{+}\cup F_{-}$.
Note that there are two type of components which are counted in $\Va(\pm \infty)$: those  which tend to a nontrivial segment of the line at infinity as the value of $f$ tends to infinity, and those which tend to a point $p\in L^{\ity}$ as the value of $f$ tends to infinity.  These two types are illustrated in Example  \ref{ex:Hlines} and Example \ref{s:exam2}.  

\begin{theorem}\label{t:main1}
\begin{equation}\label{eq:ind-sp-va}
  \ind_{\infty}(f)  = 1 + \frac12 \sum_{p\in \cL_{f}, \lambda\in \bR}\Sp(p, \lambda) - \frac12 \sum_{p\in \cL_{f}, \lambda\in \bR}\Va(p, \lambda)  - \frac12  \Va(\pm \infty)
\end{equation}
\end{theorem}
\begin{proof}
By gathering the indices of the Milnor arcs of the same cluster, one recasts \eqref{eq:indexequation} as: 
\begin{equation}\label{e:index-Cluster}
\ind _{\ity} (f) = 1 + \sum_{\cC} \sum_{\gamma \in \cC} i(\gamma),
\end{equation} 
where the sum runs over all Milnor clusters.    

Let us compute the total index $ \sum_{\gamma \in {\cC}} i(\gamma)$ for each cluster $\cC$. 
 By Corollary \ref{c:indexcluster} the number of odd clusters  of total index $+\frac{1}{2}$, or $-\frac{1}{2}$,  is equal to the number $\Sp(p,\lambda)$, or $\Va(p,\lambda)$, respectively.  Moreover, the number of clusters at $\pm \ity$ is equal to the number $\Va(\pm \ity)$ and the total index of each of these clusters is $-\frac{1}{2}$.  Even clusters have total index $0$, thus do not contribute to the formula.

Our formula \eqref{eq:ind-sp-va} follows by plugging in all these data in \eqref{e:index-Cluster}.
\end{proof}

\subsection{The local degree at infinity}\label{ss:localdeg}

We continue to consider a polynomial $f: \bR^{2}\to \bR$ of degree $d\ge 2$, and we prove here a key result that will be used for finding an upper bound of the index at infinity.
\begin{definition}\label{d:dp}
  Let $d_{Re}$ denote the number of real solutions of the equation $f_{d}=0$ counted with multiplicity. We call it the \emph{real degree}
of $f_{d}$.

 We denote by $d_{p}$ the order of $f_{d}$ at the point $p\in \{f_{d} =0\}\subset \bP^{1}$.  This is equal to the multiplicity of the linear factor of $f_{d}$ corresponding to $p$.
\end{definition} 

\begin{remark}
 The inequality $d_{p}>0$ does not imply that $p\in \cL_{f}$, like in the example $f= xy^{2}+x$, where $p:=[1:0:0]\in \{f_{d} =0\}$ with $d_{p}=2$, but $p\not\in  \cL_{f}$. 

By this reason,   out of the obvious inequalities:
\begin{equation}\label{eq:re}
d_{Re} \ge \sum_{p\in  \{ f_{d} = 0\} \cap L^{\ity} } d_{p} \ge  \sum_{p\in \cL_{f}}d_{p}
\end{equation}
 the second may be strict, for instance in the example $f = x^{4}y + y^{3}$ where $d_{Re} =5$, but one has $\sum_{p\in \cL_{f}}d_{p} =1$ because $\cL_{f} = \{ [1:0:0]\}$.  
 \end{remark}

\begin{remark}\label{r:durfeeeq}
 Let $p\in \{f_{d}=0\}\cap L^{\ity}$. 
The following equality is displayed in \cite[Lemma 7.3]{Dur2}:
\begin{equation}\label{eq:Mil}
 \mult_{p}(\overline{M_{\bC}(f)}, L^{\infty}_{\bC}) =d_p-1.
\end{equation}
We provide here an explicit proof of  \eqref{eq:Mil}.
 One may assume (by an adequate linear change of coordinates) that $p= [1:0:0]$,  and thus we have:
\[ f (x,y)= y^{d_{p}} r(x,y) + \lot \]
 where $r$ is a homogeneous polynomial of degree $ d-d_{p}$, and  not divisible by $y$.  In the chart $\{x \neq 0\}$,  the Milnor set $\overline{M_{\bC}(f)}$ has equation:
\begin{equation}\label{eq:ordery}
 \hat{h}(y,z) = - d_p y^{d_p-1} r(1,y) - y^{d_p} r_{y}(1,y)+ y^{d_p+1} r_{x}(1,y) + z q(1,y,z) =0 
\end{equation}
 where $r_x$ and $r_y $ denote the partial derivatives of $r$, and $q(x,y,z)$ is a homogeneous polynomial of degree  $d-1$.  By our assumption,  we also have $r(1,y)= c_0  + \cdots +c_ky^k $, where  $c_0 \neq 0$, and  $k\le d-d_p$.  
 One has by definition:
\[ \mult_p(\overline{M_{\bC}(f)}, L^\ity_{\bC}) = \ord\nolimits_{y}\left( \hat{h}(y,z)_{|L^\ity_{\bC}} \right) \]
 and due to \eqref{eq:ordery}, the later is precisely $d_p-1$.  
 \fin
\end{remark}
 



\section{The ``index gap'', and
upper bounds for the index at infinity}\label{upperbd}

Durfee showed in \cite{Dur2} the inequality:
\begin{equation}\label{eq:durfeeineq}
 \ind_{\infty}(f) \le 1 + d_{Re} -  2|\cL_{f}|.
\end{equation}
 We will improve  his upper bound by counting in a more refined manner the contributions
 of the Milnor branches at infinity.

Let  $p\in  L^{\ity}\cap \{f_{d} =0\}$, let $\overline{M_{\bC}(f)}\subset \bP^{2}_{\bC}$ be the projective closure  of the complex Milnor set $M_{\bC}(f)$.  We will consider  the germ $\overline{M_{\bC}(f)}_{p}$  and 
its  complex Milnor branches.
Our theorem uses the following sub-varieties of $\Cone\overline{M_{\bC}(f)}_{p}$, where the multiplicity of each line is taken  into account:

\smallskip

\noindent $\bullet$  $R_{p} := \bigl\{ L\in \Cone\overline{M_{\bC}(f)}_{p} \mid $ there is some real Milnor branch  tangent to $L$ at $p  \bigr\}$.

\noindent $\bullet$  $K_{p} :=\bigl\{ L\in \Cone\overline{M_{\bC}(f)}_{p}  \mid $ there is some complex non-real Milnor branch  tangent to $L$ at $p  \bigr\}$.

\noindent $\bullet$  $S_{p} :=\bigl\{ L\in \Cone\overline{M_{\bC}(f)}_{p}  \mid  $  $L\in R_{p}$ and either $L$ is tangent to some \emph{singular} real  branch of $\overline{M_{\bC}(f)}_{p}$, or $L$ is $ L^{\ity} \bigr\}$.

\smallskip

Note that $R_{p}^{\red} \cup S_{p}^{\red} \cup K_{p}^{\red}= \Cone\overline{M_{\bC}(f)}_{p}^{\red}$ with reduced structure, but that this union may be not disjoint.  

In order to state our index bound theorems, we denote by  $\lfloor \cdot \rfloor$ the \emph{non-negative} floor function, i.e. with the convention that if $\lfloor r \rfloor <0$, then we replace this value by $0$.

\begin{theorem}\label{t:upbound}
\begin{equation}\label{eq:upbound}
  \ind_{\infty}(f) \le 1 + d_{Re} -  2|\cL_{f}| - \sum_{p\in L^{\ity}\cap \{f_{d} =0\}} \left(\frac12 \Bigl \lfloor  \frac{\deg(R_{p}^{\red})-1}{2} \Bigr \rfloor + \deg(S_{p}) + \deg(K_{p}) \right).
\end{equation}
\end{theorem}

 The case $|\cL_{f}|=1$ is studied in detail in the next section. In order to prove Theorem \ref{t:upbound} we need the following key result.


\begin{lemma}[\textbf{The index gaps}]  \label{l:d_{p}}
 Let $p\in L^{\ity}\cap \{f_{d} =0\}$.
Then:
 \begin{equation}\label{eq:indabs}
 \sum_{c\in \bR}i_{p,c}  \le d_{p }-1.
 \end{equation}
 The following phenomena are producing the difference  between the two terms in the inequality  \eqref{eq:indabs}, to which we shall refer as ``\textbf{index gap}'':

 
\begin{enumerate}
 \rm \item \it A Milnor arc $\gamma$  at $p$ of index $i(\gamma) = - \frac12$,  yields a gap of at least $1$.
  
\rm \item \it A Milnor arc $\gamma$ at $p$ such that $\lim f_{|\gamma} = \pm \ity$, of any index,  yields a gap of at least $\frac12$.

\rm \item \it A Milnor branch $\beta$ at $p$ such that $\beta_{\bC}$ is tangent to  $L^
{\infty}_{\bC}$,  or that $\beta_{\bC}$ is singular at $p$,  yields a gap of at least $1$.  
 
 \rm \item \it A complex Milnor branch at $p$ that is not the complexified of a real Milnor branch produces a gap of at least $1$. If moreover this branch verifies the hypotheses of (c), then the gap increases to at least $2$.
 \end{enumerate}
 Moreover, the ``sign gaps'' (a), as well as the gaps (b),  cumulate with the ``singularity gaps''  (c).

\end{lemma} 

\begin{remark}\label{r:situat} 
Case (a) is illustrated by Example \ref{s:exa4}. Case (b) can be seen in Example \ref{s:exam2} in the cluster $\{ \gamma_8,\gamma_1,\gamma_2 \}$. In the same Example \ref{s:exam2}, the Milnor arcs $ \gamma_3$ and $\gamma_7$ have index $+\frac{1}{2}$ and are tangent to $L^\ity$, which means case (c). 

For case (d), let us consider the polynomial $f(x,y) =-\frac53 y^3+y^2-2y+ 4x^2+x$.  Then $p=[1:0:0] \in \overline{M_{\bC}(f)}$. The tangent cone $\Cone\overline{M_{\bC}(f)}_{p}$ is given by the equation $2z^2+6zy+5y^2 =0$, so there are at least two complex non-real branches at $p$.
\end{remark}

\begin{proof}[Proof of Lemma \ref {l:d_{p}}] 
For a fixed point $p\in \{f_{d}=0\}\cap L^{\ity}$, we have the inequalities:
\begin{equation}\label{eq:abs}
    \sum_{c\in \bR}i_{p,c} \le \sum_{c\in \bR}i_{p,c} + \sum_{c\in \{\pm \infty\}} |i_{p,c}| 
    \le  \sum_{c\in \bR\cup \{\pm \infty\}} |i_{p,c}| \le \mult_{p}(\overline{M_{\bC}(f)}, L^{\infty}_{\bC}) = d_{p}-1
\end{equation}
all of which may be strict. The first inequalities are obvious, whereas the last one is implied by the fact that there is a unique complex curve $\gamma_{\bC}$ which is the complexification of the Milnor branch  $\gamma$. The right hand side equality is \eqref{eq:Mil}.


 Each real Milnor branch at $p$ has two Milnor arcs, of indices $- \frac12$ or $+\frac12$.
The inequality \eqref{eq:indabs}  compares the indices of the \emph{real Milnor arcs} with the intersection multiplicities of the 
   corresponding complex Milnor branches.  The proof of the index gap cases goes as follows.
   
   \smallskip
   
 \noindent (a). A Milnor arc $\gamma$ at $p$ with  $i(\gamma) = -\frac12$ becomes $|i(\gamma)| = \frac12$ on the right side of \eqref{eq:abs}, which produces a gap of 1 in \eqref{eq:indabs}.


  \noindent (b). If $\gamma$ is a Milnor arc at $p$ such that $f_{|\gamma} \to \pm \ity$ then it does not exist in the sum of the left side of \eqref{eq:indabs}, whereas it contributes to the right side of \eqref{eq:abs} by $|i(\gamma)| = \frac12$.

 \noindent (c). A Milnor branch $\beta$ at $p$ such that its complexification $\beta_{\bC}$  is tangent to  $L^
{\infty}_{\bC}$  or singular at $p$ contributes by at most 1 in the left hand side of \eqref{eq:indabs}, whereas  the multiplicity $\mult_{p}(\beta_{\bC}, L^{\infty}_{\bC})$ contributes by at least 2 in the right hand side of \eqref{eq:indabs}.

\noindent (d). Any real Milnor branch has a unique complexification. However, not all local complex branches are complexifications  of real branches: there may be some purely complex Milnor branches, and all these count in the intersection index $\mult_p(\overline{M_{\bC}(f)}, L^\ity_{\bC})$, thus they give positive integer contributions in the right hand side of \eqref{eq:indabs} whereas they do not exist in the left hand side of \eqref{eq:indabs}. 

\end{proof}

\subsection{Proof of Theorem \ref{t:upbound}}
Lemma \ref{l:durfee1} reads:
\begin{equation}\label{eq:twosums}
 \ind_{\infty}(f) =1 + \sum_{p\in \cL_{f }, c\in \bR} i_{p,c} + \sum_{q\in L^{\ity}} i_{q,\infty}. 
\end{equation}
 
  By \eqref{eq:indabs}, for each $p\in L^{\ity}\cap \{f_{d} =0\}$, in particular for $p\in \cL_{f }$, we  have the inequality
 $\sum_{c\in \bR} i_{p,c} \le d_{p}-1$,
 for which we will evaluate the gaps studied in Lemma \ref{l:d_{p}}.
 
 To any complex non-real line $L\in \Cone\overline{M_{\bC}(f)}_{p}$ there corresponds a positive number of complex non-real branches having $L$ as tangent at $p$. 
 Each such complex non-real branch generates  a gap of at least 1. 
 
To any real line of $\Cone\overline{M_{\bC}(f)}_{p}$ there may correspond real and complex non-real tangent branches. 
The non-real branches contribute, by Lemma \ref{l:d_{p}}, with  gaps of at least $1$.  


A real tangent branch
$\beta$ can be either:
 
 \noindent (1).   non-singular at $p$ and  its two arcs are on both sides of $L^{\ity}$, or
 
 \noindent (2).  tangent to $L^{\ity}$  or singular at $p$, and thus yields a gap of at least 1, by Lemma \ref{l:d_{p}}(c).   According to Lemma \ref{l:d_{p}}(c), this gap cumulates with the sign gaps of type (a).
 
 \smallskip
 
 We compute a lower bound for the total gap at $p$.
 
 \smallskip
 
\noindent \textbf{The ``sign gaps'',  and the compensating exchange.} The counterclockwise ordering of the Milnor arcs (cf Definition \ref{d:orderarcs})   induces an ordering in the subset of arcs at $p$. In turn, this induces an ordering among the real semi-lines of the real tangent cone $\Cone \overline{M(f)_{p}}$ which are on the same side of the line at infinity. 

Let us assume that on the two sides of $L^{\ity}$ there are $r_{p}\ge 0$, respectively $s_{p}\ge 0$,  semi-lines with tangent real Milnor arcs, and therefore $r_{p} + s_{p} \ge \deg(R_{p}^{\red})$.  Applying Proposition \ref{l:tangent-nontangent} to the consecutive arcs on each side we obtain that there are at least $\lfloor  \frac{r_{p}}{2} \rfloor$ and $\lfloor  \frac{s_{p}}{2} \rfloor$  arcs with non-positive index, respectively. According to Lemma \ref{l:d_{p}}, each such arc produces a ``sign gap'' of at least $\frac12$.
 We thus obtain a total ``sign gap'' of at least $\frac12 (\lfloor  \frac{r_{p}}{2}\rfloor +   \lfloor  \frac{s_{p}}{2}  \rfloor )$, and notice that we have the inequality: $\lfloor  \frac{r_{p}}{2} \rfloor +   \lfloor  \frac{s_{p}}{2}  \rfloor \ge   \Bigl\lfloor  \frac{\deg(R_{p}^{\red})-1}{2} \Bigr \rfloor$.
 
Let us remark at this point that in the above computation, which starts with the ordered Milnor arcs at $p$,  to be able to apply Proposition \ref{l:tangent-nontangent} we ought to count all the Milnor arcs at $p$, thus not only the Milnor arcs for which $f$ tends to a finite value $c\in \bR$ but also the Milnor arcs\footnote{Such Milnor arcs  occur  in Example \ref{s:exam2}, see also Remark \ref{r:situat}.} for which $f$ tends to $\pm \infty$.   And since the later Milnor arcs do not occur in the sum of \eqref{eq:indabs}, we should remove them from the  above ``sign gap'' count. Nevertheless,  in the same time each of those contribute with a gap of type (b) of Lemma \ref{l:d_{p}}, and this gap of $\frac12$ is compensating 
 the necessary removal of the corresponding ``sign gap'' that has been counted as $\frac12$ too.
With this extra argument of compensating exchange,  the above estimation of the ``sign gap'' still holds.
 
 \medskip
 
 \noindent \textbf{The ``singular gaps'' and the ``complex gaps''.}
The Milnor branches which are tangent at $p$ to the lines of the sub-variety $S_{p}$ correspond to case (c) of Lemma \ref{l:d_{p}} and each one produces a gap of at least 1 which adds up to the total sign gap.  The same effect produces a complex non-real branch (by Lemma \ref{l:d_{p}}(d)), while its tangent line at $p$ belongs to $K_{p}$. In both cases, the lines are considered with their multiple structure.

By adding up those gaps,  we thus get the total gap at $p$ of at least $\frac12  \Bigl \lfloor  \frac{\deg(R_{p}^{\red})-1}{2} \Bigr \rfloor + \deg(S_{p})+ \deg(K_{p})$, i.e. the  inequality:

\begin{equation}\label{eq:localatp}
 \sum_{p\in \cL_{f }, c\in \bR} i_{p,c} \le d_{p}-1  -  \frac12  \Bigl \lfloor  \frac{\deg(R_{p}^{\red})-1}{2} \Bigr \rfloor - \deg(S_{p})- \deg(K_{p}) .
\end{equation}
 Finally we need to sum this up over all points $p\in \cL_{f }$.  Therefore we get from \eqref{eq:re} the following inequality:
\begin{equation}\label{eq:dRe}
  \sum_{p\in \cL_{f }} d_{p}-1 \le d_{Re} - |\cL_{f }|.
\end{equation}

 We have now to deal with the term $\sum_{q\in L^{\ity}} i_{q,\infty}$ in \eqref{eq:twosums}.
 By Remark \ref{r:component}, to each fibre component counted by $\Va(\pm \infty)$ there corresponds injectively a cluster $\cC = \cC(\pm \infty)$, and for this cluster the sum of indices $\sum_{\gamma \in \cC}  i(\gamma)$ is $-1/2$ by Corollary \ref{c:indexcluster}. By Proposition \ref{p:comp-at-inf}, we have $\Va(\pm \infty) \ge  2|\cL_{f}|$. We thus obtain: 
\begin{equation}\label{eq:ellminus}
 \sum_{q\in L^{\ity}} i_{q,\infty}= \sum_{\cC}\sum_{\gamma \in \cC}  i(\gamma) \le -\frac12 2|\cL_{f}| = -|\cL_{f}| 
\end{equation}
 where the first sum at the right hand side  is taken over all clusters $\cC =\cC(\pm \infty)$. 
 
 Finally,  plugging \eqref{eq:localatp}, \eqref{eq:dRe} and \eqref{eq:ellminus}  into \eqref{eq:twosums}, we obtain the claimed inequality \eqref{eq:upbound}.
 \fin
 
 \medskip
 
  We show how to compute an upper bound for the total gap in a different manner, by merging the set of singular Milnor branches (which projects onto the subset  $S_{p}^{\red}$ of $R_{p}^{\red}$) into the total set of Milnor branches (which projects onto the set  $R_{p}^{\red}$). This produces a more handy upper bound for the index at infinity, in particular formula \eqref{eq:upbound3} is in terms of the local degrees at infinity.
 
\begin{corollary}\label{c:upbound2}
 \begin{equation}\label{eq:upbound2}
   \ind_{\infty}(f) \le 1 + d_{Re} -  2|\cL_{f}| - \sum_{p\in L^{\ity}\cap \{f_{d} =0\}} \left( \Bigl \lfloor  \frac{\deg(R_{p}^{\red}) }{2} \Bigr \rfloor + \deg(K_{p}) \right).
\end{equation}
In particular:
 \begin{equation}\label{eq:upbound3}
   \ind_{\infty}(f) \le 1 + d_{Re} -  2|\cL_{f}| - \sum_{p\in L^{\ity}\cap \{f_{d} =0\}}  \Bigl \lfloor  \frac{\delta_{p}}{2} \Bigr \rfloor ,
   \end{equation}   
   where $\delta_{p} := \deg \Cone\overline{M_{\bC}(f)}_{p}^{\red_{\bR}}$ is the degree of the cone in which the real line components are taken with reduced structure.  
\end{corollary}
\begin{proof}

 For some fixed point $p\in L^{\ity}\cap \{f_{d} =0\}$, we first consider the extreme case where each $L\in R_{p}$ has some tangent real Milnor branch with arcs on both sides of the line at infinity. We thus have
  $r_{p}= s_{p}=\deg(R_{p}^{\red})$ and, by applying Proposition \ref{l:tangent-nontangent} as in the proof of Theorem \ref{t:upbound}, we get the total sign gap greater or equal to  $\Bigl\lfloor  \frac{\deg(R_{p}^{\red})}{2}  \Bigr\rfloor$. 
   
   Next we operate the following change: choose one of the Milnor branches with arcs on both sides of $L^{\ity}$ and replace it by a Milnor branch with arcs on the same side of $L^{\ity}$. Then this new branch is necessarily singular or tangent to $L^{\ity}$, thus in case (c) of Lemma \ref{l:d_{p}}  and therefore, besides its possible sign gap contribution, it gives a singular gap contribution of at least 1. Loosing one arc at one side of $L^{\ity}$
   may diminish the total gap by at most 1.
   Consequently, the effect of this replacement is that the total gap does not diminish.
   
   Finally, we observe that one may obtain any configuration of real Milnor arcs by repeating a finite number of 
   times the above described operation on Milnor branches.    
   Our first statement is proved.
   
   The second statement is a simple consequence of the inequality $ \Bigl \lfloor  \frac{\delta_{p}}{2} \Bigr \rfloor \le 
     \Bigl \lfloor  \frac{\deg(R_{p}^{\red}) }{2} \Bigr \rfloor + \deg(K_{p}) $, which follows since $\deg(R_{p}^{\red}) + \deg(K_{p}) $ is (by definition)   greater or equal to $\deg \Cone\overline{M_{\bC}(f)}_{p}^{\red_{\bR}}$.
      \end{proof}



\smallskip


\subsection{Revisiting Durfee's upper bound}\label{s:durfee}
Durfee  \cite{Dur2} has proved an upper bound \eqref{eq:durfee0} in terms of the degree $d$ only,  namely
  $\ind_\ity f \le \max\{1,  d-3\}$.    
    Revisiting and completing Durfee's proof in \cite{Dur2}, we show here the following slight improvement:
\begin{theorem}\label{t:l=1}
 Let $f: \bR^{2}\to \bR$ be a polynomial of degree $d\ge 2$ with isolated singularities.
 Then:
\begin{enumerate}
\rm \item \it  If $|\cL_{f}| \ge 2$, then $\ind_{\infty}(f) \le d-3$.
\rm \item \it  If $|\cL_{f}| =1$, then $\ind_{\infty}(f) \le d-3$ for $d \ge 4$, and $\ind_{\infty}(f) \le 0$ for $d\le 3$.
\rm \item \it  If $|\cL_{f}| =0$ then $\ind_{\infty}(f)=1$.
\end{enumerate}
\end{theorem}
\begin{proof}
(a). Follows immediately, either from  \eqref{eq:durfeeineq} or from \eqref{eq:upbound}.

\noindent (c). By Lemma \ref{l:circle}, the set $\{| f(x,y)| =R\}$ for $R\gg 1$ is diffeomorphic to a circle, and the winding number over a circle is $+1$, thus $\ind_{\infty}(f)=1$. This trivial fact was also observed in  \cite[Theorem 7.8]{Dur2}.

\noindent (b). Let $\cL_{f} = \{p\}$. Then,  by a linear change of coordinates, we may assume that $p=[1:0:0]$.
For $|\cL_{f}| =1$, by Theorem \ref{t:main1} and by Durfee's inequality \eqref{eq:durfeeineq}, or by our improvement \eqref{eq:upbound3}, we get:
\begin{equation}\label{eq:main1}
  \ind_{\infty}(f)  =  1+ \frac12 \left( \sum_{ \lambda\in \bR}\Sp(p, \lambda) -  \sum_{\lambda\in \bR\cup \{\pm \infty\}} \Va(p, \lambda)  \right) \le d_{p} - 1  
\end{equation}
If $d_{p} \le  d-2$ then Theorem \ref{t:l=1} follows directly from \eqref{eq:main1}. 
We consider in the following the  two remaining cases: $d_{p} = d-1$ and $d_{p} = d$.

\

\noindent
\frame{Case 1.} 
$d_{p} = d-1$.  
\begin{remark}\label{r:durfeeerror2}
 Durfee claims in \cite[pag. 1359]{Dur2}, this case is not possible. His argument is the following:  ``the roots of $f_d$ other than $p$ are complex, and hence occur in conjugate pairs, thus $d_{p}\le d-2$''. This seems to be grounded on the same assertion discussed in Remark \ref{r:durfeeerror1}: 
``$\cL_{f}= \emptyset \implies d_{Re}=0$'', which
 is false, as shown by the simple example $f=x^{4} +y^{2}$,  whereas its converse is obviously true.

We think that a slightly different assertion could be nevertheless true:\\ 

\noindent
\textbf{Conjecture}: 
$d_{p} = d_{Re}-1 >0\implies \cL_{f}\not= \emptyset,  \mbox{ and if moreover } p\in \cL_{f} \mbox{ then } |\cL_{f}| =2.$
 
 \
 
 We will prove by the next lemma a version of this conjecture with $d_{Re}$ replaced by $d$.  
\end{remark}
\begin{lemma}\label{l:durfee3}
 If $d_{p} = d-1 >0$  then  $\cL_{f}\not= \emptyset$,   and if moreover  $p\in \cL_{f}$  then  $|\cL_{f}| =2$.
\end{lemma}
\begin{proof}
  If $d_{p} = d-1 >0$  then one may assume, by an appropriate linear change of variables,
 that $p = [1:0:0]$ and $f_{d}=xy^{d-1}$. One therefore has: 
\[
f(x,y)= xy^{d-1}+ xh(x,y) +  u(y)
\] 
with $\deg h(x,y) < d-2$  and $\deg u(y)\le d-1$. 

We will show that the projective closure $\overline{f^{-1}(0)}$ contains the point $[0:1:0] \in L^{\ity}$, which is different from $p$.  We may assume that  $u\not \equiv 0$, since if not, then  $\{x=0\} \subset f^{-1}(0)$, thus  $[0:1:0]\in \cL_{f}$ and the claim is proved. Let then $b\in \bR^{*}$ be the leading coefficient of $u$.

For every fixed $x_0 \in \bR^{*}$, $x_{0} \not= -b$, the sign of the polynomial $f(x_0,y)$ of variable $y$, for $y\gg 1$, is the sign of its leading coefficient; this is $x_0+b$ in case $\deg u=d-1$,  and it is  $x_0$ in case $\deg u < d-1$.  Let us choose $x_0 \in \bR^*$ such that $b (x_0 +b)<0$, and thus $b x_0<0$ too. Then $f(x_0,y)f(0,y)<0$ for any $y\gg 1$.  This implies that for any $y \gg 1$ there exists some value $t_{y}$ bounded between $0$ and $x_{0}$ such that $(t_{y}, y)\in  f^{-1}(0)$. By taking the limit $y\to \ity$, this shows that $[0:1:0]\in \overline{f^{-1}(0)} \cap L^{\ity}$.  
\end{proof}

 Lemma \ref{l:durfee3} shows that the case $d= d_{p} - 1$, with $p\in \cL_{f}$ and $|\cL_{f}|=1$ is impossible, confirming Durfee's claim.

\medskip
\noindent
\frame{Case 2.}  $d_{p} = d$.

We may assume as above that $p=[1:0:0]$ and since $d_{p} = d$, one may also assume that, modulo some appropriate linear change of coordinates, one has $f_{d} =y^{d}$.


\begin{lemma}\label{l:separate}
Let $p= [1:0:0] \in \cL_{f}$ and let  $f_{d} =  y^{d}$.
Then either there exists a Milnor branch at $p$ which is tangent
to  $L^{\infty}$, or 
 $f(x,y) = y^{d} + v(x) + u(y)$, where $\deg v \le 2 <d$ and $\deg u \le d-1$.
\end{lemma}
\begin{proof}
Let us assume that  $f$  contains mixed terms, namely let $f = y^{d}+ xyq(x,y) +v(x) + u(y)$, where $q \not\equiv 0$ is a polynomial of degree $\le d-3$, and
$v(x)$ and  $u(y)$ are some polynomials of degrees $\le d-1$.
We write explicitly the equation  $\hat h(1,y,z) =0$ of the closure $\overline{M(f)}$ of  the Milnor set in the chart $\{x=1\}$. This is a polynomial of degree at least $d-1$ because it contains the term  $dy^{d-1}$. Its order at $(0,0)$ is $\ord \hat h(1,y,z) \le \ord  z\hat q(1,y,z) + zy\hat q(1,y,z) \le d-2$, where $\hat q(x,y,z)$
denotes the homogenization of $q(x,y)$  of degree $d-3$ by the variable $z$. Thus all the terms of $\hat h(1,y,z)$ of degree $< d-1$ contain $z$. This implies that $\overline{M(f)}_{p}$
 contains a (complex) branch which is tangent to the line at infinity $L^{\infty} = \{z=0\}$.

Let us now treat the complementary case, i.e. whenever $f$ contains no mixed terms, i.e. $f= y^{d} +v(x) + u(y)$, where $v(x)$ and  $u(y)$ are some polynomials of degrees $\le d-1$. Then the Milnor set germ $\overline{M(f)}_{p}$ in the chart $\{x =1\}$  is defined by the equation: 
\begin{equation}\label{e:Milg+u}
\hat{h}(1,y,z)= - d  y^{d-1} +  zy\hat{v}_x(1,z)  - z \hat{u}_y(y,z),
\end{equation}
where $\hat{v}_x(x,z)$ and $\hat{u}_y(y,z)$ denote the homogenization of degree $d-2$ of the derivatives $v_x$ and $u_y$.
The 1st and the 3rd terms of \eqref{e:Milg+u} are homogeneous of degree $d-1$, while the term in the middle is of order $\le d-2$ if and only if 
$\deg v \ge 3$. We deduce that the tangent cone $\Cone _{p}\overline{M(f)}$ contains the line $\{z=0\}$ if and only if  $\deg v \ge 3$. 
\end{proof}

We compute the index in the special case of Lemma \ref{l:separate}.
\begin{lemma}\label{l:indexspec}
Let $f=v(x) + u(y)$, where $\deg u =\deg f =d$, and $\deg v \le 2< d$ such that $v_{x} \not\equiv 0$. 
If $d \geq 3$ then $|\ind _\ity(f)|$ is $0$ or $1$.
\end{lemma}
\begin{proof}
If $\deg v \le 1$ then $\Sing f = \emptyset$ and therefore $\ind _\ity(f) =0$.

If $\deg v =2$, then the derivative $v_{x} = ax+b$, where $a\not= 0$ by our assumption, 
 changes sign one time, precisely at $x = -b/a$.
 Consider a large enough circle $C\subset \bR^{2}$ centred at the origin.  Consider the two points $N, S \in C\cap \{x = -b/a\}$.
 In the following we will compute the index at infinity 
 $\ind_{\ity} (f)$ as the winding number over $C$. 
  
 If $d \geq 3$ then $u_{y}$ is a polynomial of degree $d-1$ and therefore has a constant sign outside a compact 
subset of $\bR$. This implies that on each half circle of $C$  cut out by the vertical line $\{x = -b/a\}$, the variation of 
the vector field $\grad f$ over the circle $C$ between the two points $N$ and $S$ is either zero, or $\pi$ or $-\pi$ 
\end{proof} 
We continue the proof of Case 2.  By Lemma \ref{l:separate} and by Lemma \ref{l:indexspec}, if there is  no Milnor branch tangent to $L^{\infty}$ at $p$, then we obtain $\ind_{\infty}(f) =0$ when $d =3$, and $\ind_{\infty}(f) \le d-3$ when $d>3$, hence Theorem \ref{t:l=1} is proved in this situation.

In what follows we focus on the last remaining case established by Lemma \ref{l:separate}, namely:
 \emph{there exists at least one (complex) Milnor branch $\beta$ tangent to   $L^{\infty}$ at $p$.}
The study falls into the following 4 situations: 

\

\noindent  (i).
There are at least two Milnor branches at $p$ which are tangent to $L^{\ity}$, then by Lemma \ref{l:d_{p}}(c) we get a gap of at least 1 for each of these branches. Therefore $\ind_{\infty}(f) \le d_{p}-1-2=  d_{p} - 3$.

\

\noindent (ii). There is a single Milnor branch $\beta$ tangent to $L^{\ity}$ at $p$,  and such that $\mult_{p}(\beta, L^{\infty}_{\bC}) \ge 2$. By \eqref{eq:Mil}, this implies $d_{p}\ge 3$.

 If $\mult_{p}(\beta, L^{\infty}_{\bC}) >2$ then we have a gap of at least 2, and thus $\ind_{\infty}(f) \le d_{p}-3$.
 If $\mult_{p}(\beta, L^{\infty}_{\bC}) = 2$  and the indices of the arcs of $\beta$ are not both $+\frac12$, then we get a gap of at least $3/2$. Since the index is an integer, the gap is of  at least 2, and therefore $\ind_{\infty}(f) \le d_{p}-3$ again.

Last case, let $\mult_{p}(\beta, L^{\infty}_{\bC}) = 2$ and such that both arcs of $\beta$ have  index $+\frac12$.
Since  $\beta$ is also tangent to $L^{\ity}$,  it follows that $\beta$ is nonsingular at $p$, more precisely the germ of $\beta$ at $p$ is equivalent, after some linear change of coordinates, with the curve $z=y^{2}$. Thus the two Milnor arcs are in the same half-plane of the chart $\bR^{2}$ cut by the line $z=0$.  According to Proposition \ref{p:tangency}, after each of the two splittings we obtain two components of the fibres of $f$ which are tangent to the line $L := \{z=0\}$, and moreover,  along the two Milnor arcs the tangency is to different  semi-lines, say  $L_{+}$ and $L_{-}$. This implies that, in the absence of 
other splittings or vanishings at $p$, there should exist a trivial fibration connecting two components tangent to different 
semi-lines, which is treated by Proposition \ref{l:tangent-nontangent}, and which tells that this situation is impossible.

\

\noindent  (iii). There is a single Milnor branch  tangent to $L^{\ity}$,
and there are also non-tangent Milnor branches at $p$, such that at least one of the Milnor arcs has index $-\frac12 $.
Then by 
Lemma \ref{l:d_{p}}(a) and (c) we obtain an index gap of at least 2. Therefore we get $\ind_{\infty}(f) \le d_{p}-3$.  
 
 \
 
\noindent  (iv).  There is a single tangent Milnor branch  at $p$, there are one or more transversal Milnor branches,  and such that all the Milnor arcs at $p$ have index $+\frac12$. This means that all arcs are of splitting type, and in particular each arc is a cluster. Our Proposition \ref {l:tangent-nontangent} shows that this situation is impossible.
 \end{proof}



\section{Examples}\label{examples}
\sloppy

We consider here four examples. For three of them we will use pictures to encode information, and in order to draw the frame we  will use the following construction employed  in \cite{Duall}.
 Let  $\bR^2 \hookrightarrow \bP^2 \simeq  \bR^3\m \{0\} / \bR^{*}$ be the embedding defined by $(x,y)\mapsto [x:y:1]$, and let
\[S := \{ (a,b,0) \in \bR^3\m \{0\} \}/ \bR_{+} ,\] 
be the  circle which is a double covering of the line at infinity $L^\infty \subset \bP^2$.
The compactification $\bR^2 \sqcup S$ of $\bR^{2}$ may be represented as a 2-disk $D$ with boundary  $S$.  

 The dashed circle is the boundary $\partial D_{R}$ of the disk $D_{R}$  centred at the origin of radius $R\gg 1$ as in Proposition \ref{p:milnorarcs}. The Milnor arcs  live in the annulus between $\partial D_{R}$ and $S$. By enumerating the Milnor arcs as $\gamma_1,\ldots , \gamma_k$ we mean that they are consecutive in the counterclockwise ordering (Definition \ref{d:orderarcs}).
 
 The limit $\lambda\in \bR \cup \{\pm \ity\}$ to which $f_{|\gamma}$ tends along some Milnor arc $\gamma$ is written near the Milnor arc $\gamma$ close to $S$ (see Proposition \ref{p:milnorarcs}).  The index $i(\gamma)$ is attached to each Milnor arc $\gamma$ at the intersection with the doted circle in the middle of the annulus, and
  the respective little arrow indicates the direction of the gradient of the Milnor arc.   We write ``$\Sp$'' or ``$\Va$'' next to a cluster when the corresponding fibre component is splitting or vanishing, respectively. Whenever a cluster contains more than one Milnor arc,  we connect all its arcs by a thicker curve (e.g. like in Figure \ref{fig:ex2}).

\medskip

\begin{example}\label{ex:broughton}
Let  $f(x,y)=x^2y+x$.  This polynomial has  two clusters at the atypical value $\lambda=0$, each of them composed by a single Milnor arc of positive index. Both clusters have the point $p=[0:1:0]$ at infinity, and no other Milnor arcs abut to this point. Thus  $d_p=2$, and \eqref{eq:indabs} is an equality.

The polynomial $f$ has two clusters having the point $q=[1:0:0]$ at infinity,  the corresponding fibre components of which tend to the value $+\ity$. And there are two more clusters with the corresponding fibre components  tending to the value $-\ity$.   These 4 clusters being vanishing clusters at $\lambda = \pm \infty$, all of them have index $-\frac 12$ by Corollary \ref{c:indexcluster}.

 It then follows from Theorem \ref{t:main1}  that $\ind_{\ity}f = 1+2\cdot \frac 12  - 4\cdot \frac 12 =0$.    Comparing to Lemma \ref{l:d_{p}},  there are no index gaps of any kind, and  this example realises the maximal index at infinity that a polynomial of degree $3$ may have, cf Theorem \ref{t:l=1}(a).
\end{example}

 \medskip

\begin{example}\label{exa1}
Let $f(x,y)=y^5+x^2y^3-y$.
The Milnor set $M(f)$ is defined by the equation 
 $x (-1 + 3 x^2 y^2 + 3 y^4)=0$.

\begin{figure}[ht!]
\begin{center}

\tikzset{every picture/.style={line width=0.75pt}} 

\begin{tikzpicture}[x=0.60pt,y=0.60pt,yscale=-1,xscale=1]

\draw   (329.47,29.92) .. controls (401.03,30.55) and (459.31,89.06) .. (459.63,160.62) .. controls (459.96,232.17) and (402.21,289.67) .. (330.64,289.05) .. controls (259.08,288.43) and (200.8,229.91) .. (200.47,158.36) .. controls (200.15,86.8) and (257.9,29.3) .. (329.47,29.92) -- cycle ;
\draw  [dash pattern={on 4.5pt off 4.5pt}] (285.26,159.49) .. controls (285.26,134.75) and (305.32,114.69) .. (330.05,114.69) .. controls (354.79,114.69) and (374.85,134.75) .. (374.85,159.49) .. controls (374.85,184.22) and (354.79,204.28) .. (330.05,204.28) .. controls (305.32,204.28) and (285.26,184.22) .. (285.26,159.49) -- cycle ;
\draw  [dash pattern={on 0.84pt off 2.51pt}] (236.83,159.49) .. controls (236.83,108) and (278.57,66.26) .. (330.05,66.26) .. controls (381.54,66.26) and (423.28,108) .. (423.28,159.49) .. controls (423.28,210.97) and (381.54,252.71) .. (330.05,252.71) .. controls (278.57,252.71) and (236.83,210.97) .. (236.83,159.49) -- cycle ;
\draw [color={rgb, 255:red, 0; green, 0; blue, 0 }  ,draw opacity=1 ]   (330.05,114.69) .. controls (331,115.67) and (317.67,49) .. (329.47,29.92) ;
\draw [color={rgb, 255:red, 0; green, 0; blue, 0 }  ,draw opacity=1 ]   (201.5,153.25) .. controls (209.5,154.25) and (230.5,141.25) .. (247.5,137.25) .. controls (264.5,133.25) and (277.08,135.04) .. (288.67,140.83) ;
\draw [color={rgb, 255:red, 0; green, 0; blue, 0 }  ,draw opacity=1 ]   (201.5,153.25) .. controls (216.5,156.25) and (259,186.75) .. (287.33,176.5) ;
\draw [color={rgb, 255:red, 0; green, 0; blue, 0 }  ,draw opacity=1 ]   (323.67,287.67) .. controls (338.33,241) and (321,220.33) .. (330.05,204.28) ;
\draw [color={rgb, 255:red, 0; green, 0; blue, 0 }  ,draw opacity=1 ]   (459,156.25) .. controls (432.5,158.25) and (397.5,179.75) .. (371.5,174.75) ;
\draw [color={rgb, 255:red, 0; green, 0; blue, 0 }  ,draw opacity=1 ]   (459,156.25) .. controls (440,152.25) and (424,140.25) .. (407.5,137.75) .. controls (391,135.25) and (372.67,134.33) .. (369,138.75) ;
\draw    (237.83,170.33) -- (253.5,170.7) ;
\draw [shift={(255.5,170.75)}, rotate = 181.35] [color={rgb, 255:red, 0; green, 0; blue, 0 }  ][line width=0.75]    (6.56,-1.97) .. controls (4.17,-0.84) and (1.99,-0.18) .. (0,0) .. controls (1.99,0.18) and (4.17,0.84) .. (6.56,1.97)   ;
\draw    (330.05,252.71) -- (332.77,240.35) ;
\draw [shift={(333.2,238.4)}, rotate = 102.4] [color={rgb, 255:red, 0; green, 0; blue, 0 }  ][line width=0.75]    (6.56,-1.97) .. controls (4.17,-0.84) and (1.99,-0.18) .. (0,0) .. controls (1.99,0.18) and (4.17,0.84) .. (6.56,1.97)   ;
\draw    (422.05,142.88) -- (436.5,141.89) ;
\draw [shift={(438.5,141.75)}, rotate = 176.08] [color={rgb, 255:red, 0; green, 0; blue, 0 }  ][line width=0.75]    (6.56,-1.97) .. controls (4.17,-0.84) and (1.99,-0.18) .. (0,0) .. controls (1.99,0.18) and (4.17,0.84) .. (6.56,1.97)   ;
\draw    (324.15,66.21) -- (322.61,80.81) ;
\draw [shift={(322.4,82.8)}, rotate = 276.03] [color={rgb, 255:red, 0; green, 0; blue, 0 }  ][line width=0.75]    (6.56,-1.97) .. controls (4.17,-0.84) and (1.99,-0.18) .. (0,0) .. controls (1.99,0.18) and (4.17,0.84) .. (6.56,1.97)   ;
\draw    (422.67,166.17) -- (410,166.24) ;
\draw [shift={(408,166.25)}, rotate = 359.67] [color={rgb, 255:red, 0; green, 0; blue, 0 }  ][line width=0.75]    (6.56,-1.97) .. controls (4.17,-0.84) and (1.99,-0.18) .. (0,0) .. controls (1.99,0.18) and (4.17,0.84) .. (6.56,1.97)   ;
\draw    (238.83,140.83) -- (228.5,140.27) ;
\draw [shift={(226.5,140.17)}, rotate = 3.09] [color={rgb, 255:red, 0; green, 0; blue, 0 }  ][line width=0.75]    (6.56,-1.97) .. controls (4.17,-0.84) and (1.99,-0.18) .. (0,0) .. controls (1.99,0.18) and (4.17,0.84) .. (6.56,1.97)   ;

\draw (327.47,32.57) node [anchor=north west][inner sep=0.75pt]  [font=\tiny]  {$+\infty $};
\draw (326.63,275) node [anchor=north west][inner sep=0.75pt]  [font=\tiny]  {$-\infty $};
\draw (350,136.23) node [anchor=north west][inner sep=0.75pt]  [font=\tiny]  {$\gamma _{1}$};
\draw (324.83,120.23) node [anchor=north west][inner sep=0.75pt]  [font=\tiny]  {$\gamma _{2}$};
\draw (294.5,141.73) node [anchor=north west][inner sep=0.75pt]  [font=\tiny]  {$\gamma _{3}$};
\draw (291.17,168.23) node [anchor=north west][inner sep=0.75pt]  [font=\tiny]  {$\gamma _{4}$};
\draw (326.33,191.07) node [anchor=north west][inner sep=0.75pt]  [font=\tiny]  {$\gamma _{5}$};
\draw (350,168.23) node [anchor=north west][inner sep=0.75pt]  [font=\tiny]  {$\gamma _{6}$};
\draw (305.3,13.13) node [anchor=north west][inner sep=0.75pt]  [font=\tiny]  {$[0:1:0]$};
\draw (466.5,149) node [anchor=north west][inner sep=0.75pt]  [font=\tiny]  {$[1:0:0]$};
\draw (329.5,69.9) node [anchor=north west][inner sep=0.75pt]  [font=\fontsize{0.41em}{0.49em}\selectfont]  {$-\frac{1}{2}$};
\draw (305,229.73) node [anchor=north west][inner sep=0.75pt]  [font=\fontsize{0.41em}{0.49em}\selectfont]  {$-\frac{1}{2}$};
\draw (247.81,119.76) node [anchor=north west][inner sep=0.75pt]  [font=\fontsize{0.41em}{0.49em}\selectfont]  {$+\frac{1}{2}$};
\draw (390,120) node [anchor=north west][inner sep=0.75pt]  [font=\fontsize{0.41em}{0.49em}\selectfont]  {$+\frac{1}{2}$};
\draw (368.5,120) node [anchor=north west][inner sep=0.75pt]  [font=\tiny]  {$Sp$};
\draw (375.85,160) node [anchor=north west][inner sep=0.75pt]  [font=\tiny]  {$Sp$};
\draw (275,182.4) node [anchor=north west][inner sep=0.75pt]  [font=\tiny]  {$Sp$};
\draw (265,140.9) node [anchor=north west][inner sep=0.75pt]  [font=\tiny]  {$Sp$};
\draw (395,176.4) node [anchor=north west][inner sep=0.75pt]  [font=\fontsize{0.41em}{0.49em}\selectfont]  {$+\frac{1}{2}$};
\draw (247.69,179.56) node [anchor=north west][inner sep=0.75pt]  [font=\fontsize{0.41em}{0.49em}\selectfont]  {$+\frac{1}{2}$};
\draw (445,161.4) node [anchor=north west][inner sep=0.75pt]  [font=\tiny]  {$0$};
\draw (445,140) node [anchor=north west][inner sep=0.75pt]  [font=\tiny]  {$0$};
\draw (203.63,135) node [anchor=north west][inner sep=0.75pt]  [font=\tiny]  {$0$};
\draw (202.47,165) node [anchor=north west][inner sep=0.75pt]  [font=\tiny]  {$0$};
\draw (140,148) node [anchor=north west][inner sep=0.75pt]  [font=\tiny]  {$[1:0:0]$};
\draw (308.5,299.4) node [anchor=north west][inner sep=0.75pt]  [font=\tiny]  {$[0:1:0]$};

\end{tikzpicture}
\end{center}
\caption{Milnor arcs of $f= y^5+x^2y^3-y$}\label{fig:ex1}
\end{figure}
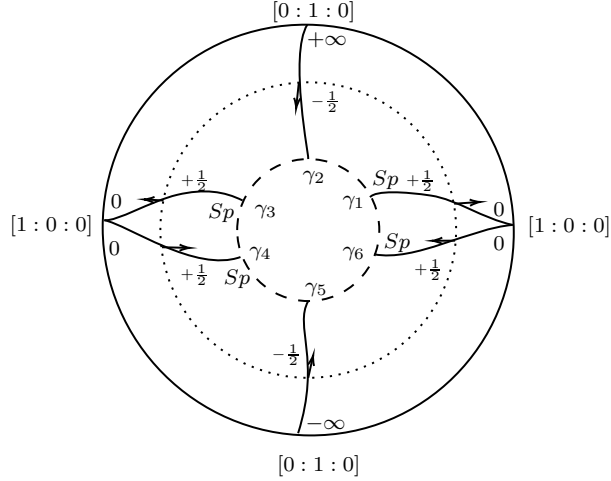

\end{example}

We have $d=5$, $\cL_{f} = \{p\}$ with $p:=[1:0:0]$, and there are two other complex non-real points at infinity due to the 
factor $y^2+x^2$ of  the top homogeneous part $f_{5}$.
In the chart  $\{x=1\}$ of  $\bP^2$, the germ  at $p$ of the Milnor set $\overline{M(f)}$ is defined by the equation
$\hat{h}(y,z)=-z^4+3y^2+3y^4=0$, thus
$\Cone \overline{M_{\bC}(f)}_{p}=2 \{y=0\}$. There are 4 clusters having the point $p \in \cL_f$ at infinity, each containing a single Milnor arc, all being splitting clusters at the value $0$, and one pair of clusters is tangent to a semi-line, and the second pair of clusters is tangent to the other semi-line. Compare also to Proposition \ref{l:tangent-nontangent}.


The fibres at infinity $F_+ \cup F_-$ consist of two components; one corresponds to the cluster $\{\gamma_2\}$, see Figure \ref{fig:ex1}, and covers the upper semi-circle of $S$, and the other corresponds to the cluster $\{\gamma_5\}$ and covers the lower semi-circle of $S$.
 By Theorem \ref{t:main1}, one then has $\ind_{\ity}(f)=2$, which is the highest possible index at infinity of a polynomial of degree $d=5$ with $|\cL_f| =1$, according to Theorem \ref{t:l=1}(b).
Since $d_{p} = d_{Re} = 3$, $\deg (R_{p}^{\red})=1$,  $S_{p}= \emptyset$, and $K_{p}=\emptyset$, the inequality \eqref{eq:upbound} reads:
\[\ind_\ity (f) \le 1 + 3 - 2  = 2.\]
This is an equality in our case, and the same are \eqref{eq:indabs},   \eqref{eq:upbound2} and \eqref{eq:upbound3}.

\medskip 

\begin{example}\label{s:exam2}
Let  $f(x,y)=(x-y^2)\Bigl( (x-y^2)(y^2+1)-1\Bigr)$. 

Its Milnor set is defined by the equation:
$y (1 - 4 x^2 + 2 x^3 + 2 y^2 + 2 x y^2 - 8 x^2 y^2 + 2 y^4 + 6 x y^4)=0.$

\begin{figure}[ht!]
\begin{center}

\tikzset{every picture/.style={line width=0.75pt}} 

\begin{tikzpicture}[x=0.60pt,y=0.60pt,yscale=-1,xscale=1]

\draw   (329.47,29.92) .. controls (401.03,30.55) and (459.31,89.06) .. (459.63,160.62) .. controls (459.96,232.17) and (402.21,289.67) .. (330.64,289.05) .. controls (259.08,288.43) and (200.8,229.91) .. (200.47,158.36) .. controls (200.15,86.8) and (257.9,29.3) .. (329.47,29.92) -- cycle ;
\draw  [dash pattern={on 4.5pt off 4.5pt}] (285.26,159.49) .. controls (285.26,134.75) and (305.32,114.69) .. (330.05,114.69) .. controls (354.79,114.69) and (374.85,134.75) .. (374.85,159.49) .. controls (374.85,184.22) and (354.79,204.28) .. (330.05,204.28) .. controls (305.32,204.28) and (285.26,184.22) .. (285.26,159.49) -- cycle ;
\draw  [dash pattern={on 0.84pt off 2.51pt}] (236.83,159.49) .. controls (236.83,108) and (278.57,66.26) .. (330.05,66.26) .. controls (381.54,66.26) and (423.28,108) .. (423.28,159.49) .. controls (423.28,210.97) and (381.54,252.71) .. (330.05,252.71) .. controls (278.57,252.71) and (236.83,210.97) .. (236.83,159.49) -- cycle ;
\draw [color={rgb, 255:red, 0; green, 0; blue, 0 }  ,draw opacity=1 ]   (330.05,114.69) .. controls (331,115.67) and (317.67,49) .. (329.47,29.92) ;
\draw [color={rgb, 255:red, 0; green, 0; blue, 0 }  ,draw opacity=1 ]   (323.67,287.67) .. controls (338.33,241) and (321,220.33) .. (330.05,204.28) ;
\draw [color={rgb, 255:red, 0; green, 0; blue, 0 }  ,draw opacity=1 ]   (459,156.25) .. controls (441.2,179.4) and (397.5,179.75) .. (371.5,174.75) ;
\draw [color={rgb, 255:red, 0; green, 0; blue, 0 }  ,draw opacity=1 ]   (459,156.25) .. controls (444.4,129) and (387.6,125.4) .. (369,138.75) ;
\draw    (330.05,252.71) -- (332.1,266.42) ;
\draw [shift={(332.4,268.4)}, rotate = 261.5] [color={rgb, 255:red, 0; green, 0; blue, 0 }  ][line width=0.75]    (6.56,-1.97) .. controls (4.17,-0.84) and (1.99,-0.18) .. (0,0) .. controls (1.99,0.18) and (4.17,0.84) .. (6.56,1.97)   ;
\draw    (415.25,196.18) -- (430,196.73) ;
\draw [shift={(432,196.8)}, rotate = 182.14] [color={rgb, 255:red, 0; green, 0; blue, 0 }  ][line width=0.75]    (6.56,-1.97) .. controls (4.17,-0.84) and (1.99,-0.18) .. (0,0) .. controls (1.99,0.18) and (4.17,0.84) .. (6.56,1.97)   ;
\draw    (324.55,66.21) -- (320.24,53.49) ;
\draw [shift={(319.6,51.6)}, rotate = 71.27] [color={rgb, 255:red, 0; green, 0; blue, 0 }  ][line width=0.75]    (6.56,-1.97) .. controls (4.17,-0.84) and (1.99,-0.18) .. (0,0) .. controls (1.99,0.18) and (4.17,0.84) .. (6.56,1.97)   ;
\draw    (421.57,175.37) -- (434.4,175.91) ;
\draw [shift={(436.4,176)}, rotate = 182.44] [color={rgb, 255:red, 0; green, 0; blue, 0 }  ][line width=0.75]    (6.56,-1.97) .. controls (4.17,-0.84) and (1.99,-0.18) .. (0,0) .. controls (1.99,0.18) and (4.17,0.84) .. (6.56,1.97)   ;
\draw [color={rgb, 255:red, 0; green, 0; blue, 0 }  ,draw opacity=1 ]   (359.25,126) .. controls (399.25,96) and (458,129.8) .. (459,156.25) ;
\draw    (412.05,116.58) -- (427.6,117.48) ;
\draw [shift={(429.6,117.6)}, rotate = 183.34] [color={rgb, 255:red, 0; green, 0; blue, 0 }  ][line width=0.75]    (6.56,-1.97) .. controls (4.17,-0.84) and (1.99,-0.18) .. (0,0) .. controls (1.99,0.18) and (4.17,0.84) .. (6.56,1.97)   ;
\draw [color={rgb, 255:red, 0; green, 0; blue, 0 }  ,draw opacity=1 ]   (362.75,189.75) .. controls (386.75,211.5) and (461.6,191) .. (459,156.25) ;
\draw [color={rgb, 255:red, 0; green, 0; blue, 0 }  ,draw opacity=1 ]   (375.25,156.5) -- (459,156.25) ;
\draw    (423.6,156) -- (435.25,153.26) ;
\draw [shift={(437.2,152.8)}, rotate = 166.76] [color={rgb, 255:red, 0; green, 0; blue, 0 }  ][line width=0.75]    (6.56,-1.97) .. controls (4.17,-0.84) and (1.99,-0.18) .. (0,0) .. controls (1.99,0.18) and (4.17,0.84) .. (6.56,1.97)   ;
\draw    (421.47,133.57) -- (433.6,133.6) ;
\draw [shift={(435.6,133.6)}, rotate = 180.14] [color={rgb, 255:red, 0; green, 0; blue, 0 }  ][line width=0.75]    (6.56,-1.97) .. controls (4.17,-0.84) and (1.99,-0.18) .. (0,0) .. controls (1.99,0.18) and (4.17,0.84) .. (6.56,1.97)   ;
\draw [color={rgb, 255:red, 0; green, 0; blue, 0 }  ,draw opacity=1 ]   (200.47,158.36) -- (285.26,159.49) ;
\draw  [draw opacity=0][line width=1.5]  (399,182.05) .. controls (407.74,181.17) and (414.65,169.73) .. (414.65,155.75) .. controls (414.65,141.27) and (407.24,129.53) .. (398.07,129.4) -- (397.93,155.75) -- cycle ; \draw  [line width=1.5]  (399,182.05) .. controls (407.74,181.17) and (414.65,169.73) .. (414.65,155.75) .. controls (414.65,141.27) and (407.24,129.53) .. (398.07,129.4) ;  
\draw    (236.83,159.49) -- (221.98,161.7) ;
\draw [shift={(220,162)}, rotate = 351.5] [color={rgb, 255:red, 0; green, 0; blue, 0 }  ][line width=0.75]    (6.56,-1.97) .. controls (4.17,-0.84) and (1.99,-0.18) .. (0,0) .. controls (1.99,0.18) and (4.17,0.84) .. (6.56,1.97)   ;
\draw  [draw opacity=0][line width=1.5]  (347,220.31) .. controls (340.1,222.54) and (332.7,223.75) .. (325,223.75) .. controls (287.58,223.75) and (257.25,195.15) .. (257.25,159.88) .. controls (257.25,124.6) and (287.58,96) .. (325,96) .. controls (332.7,96) and (340.1,97.21) .. (347,99.44) -- (325,159.88) -- cycle ; \draw  [line width=1.5]  (347,220.31) .. controls (340.1,222.54) and (332.7,223.75) .. (325,223.75) .. controls (287.58,223.75) and (257.25,195.15) .. (257.25,159.88) .. controls (257.25,124.6) and (287.58,96) .. (325,96) .. controls (332.7,96) and (340.1,97.21) .. (347,99.44) ;  

\draw (327.47,32.57) node [anchor=north west][inner sep=0.75pt]  [font=\tiny]  {$+\infty $};
\draw (326,275) node [anchor=north west][inner sep=0.75pt]  [font=\tiny]  {$+\infty $};
\draw (355,151.73) node [anchor=north west][inner sep=0.75pt]  [font=\tiny]  {$\gamma _{1}$};
\draw (350,136.23) node [anchor=north west][inner sep=0.75pt]  [font=\tiny]  {$\gamma _{2}$};
\draw (340,124.73) node [anchor=north west][inner sep=0.75pt]  [font=\tiny]  {$\gamma _{3}$};
\draw (325.17,118.73) node [anchor=north west][inner sep=0.75pt]  [font=\tiny]  {$\gamma _{4}$};
\draw (290.33,155.57) node [anchor=north west][inner sep=0.75pt]  [font=\tiny]  {$\gamma _{5}$};
\draw (324.33,192.73) node [anchor=north west][inner sep=0.75pt]  [font=\tiny]  {$\gamma _{6}$};
\draw (330.45,71.26) node [anchor=north west][inner sep=0.75pt]  [font=\fontsize{0.29em}{0.35em}\selectfont]  {$-\frac{1}{2}$};
\draw (332,230) node [anchor=north west][inner sep=0.75pt]  [font=\fontsize{0.29em}{0.35em}\selectfont]  {$-\frac{1}{2}$};
\draw (236,165) node [anchor=north west][inner sep=0.75pt]  [font=\fontsize{0.29em}{0.35em}\selectfont]  {$+\frac{1}{2}$};
\draw (375,115.1) node [anchor=north west][inner sep=0.75pt]  [font=\fontsize{0.29em}{0.35em}\selectfont]  {$-\frac{1}{2}$};
\draw (356,105) node [anchor=north west][inner sep=0.75pt]  [font=\tiny]  {$\text{Sp}$};
\draw (359.35,196.39) node [anchor=north west][inner sep=0.75pt]  [font=\tiny]  {$\text{Sp}$};
\draw (440,188) node [anchor=north west][inner sep=0.75pt]  [font=\tiny]  {$0$};
\draw (440,115) node [anchor=north west][inner sep=0.75pt]  [font=\tiny]  {$0$};
\draw (375,95) node [anchor=north west][inner sep=0.75pt]  [font=\fontsize{0.29em}{0.35em}\selectfont]  {$+\frac{1}{2}$};
\draw (385,135) node [anchor=north west][inner sep=0.75pt]  [font=\fontsize{0.29em}{0.35em}\selectfont]  {$+\frac{1}{2}$};
\draw (385,160.4) node [anchor=north west][inner sep=0.75pt]  [font=\fontsize{0.29em}{0.35em}\selectfont]  {$-\frac{1}{2}$};
\draw (380,200) node [anchor=north west][inner sep=0.75pt]  [font=\fontsize{0.29em}{0.35em}\selectfont]  {$+\frac{1}{2}$};
\draw (200.33,145) node [anchor=north west][inner sep=0.75pt]  [font=\tiny]  {$+\infty $};
\draw (350,167.73) node [anchor=north west][inner sep=0.75pt]  [font=\tiny]  {$\gamma _{8}$};
\draw (345,181.73) node [anchor=north west][inner sep=0.75pt]  [font=\tiny]  {$\gamma _{7}$};
\draw (419,139.27) node [anchor=north west][inner sep=0.75pt]  [font=\tiny]  {$+\infty $};
\draw (420,157.47) node [anchor=north west][inner sep=0.75pt]  [font=\tiny]  {$+\infty $};
\draw (465.4,154.2) node [anchor=north west][inner sep=0.75pt]  [font=\tiny]  {$[1:0:0]$};
\draw (140,155.8) node [anchor=north west][inner sep=0.75pt]  [font=\tiny]  {$[1:0:0]$};
\draw (308,11) node [anchor=north west][inner sep=0.75pt]  [font=\tiny]  {$[ 0:1:0]$};
\draw (311.4,296.6) node [anchor=north west][inner sep=0.75pt]  [font=\tiny]  {$[ 0:1:0]$};
\draw (415,50) node [anchor=north west][inner sep=0.75pt]  [font=\tiny]  {S};

\end{tikzpicture}

\caption{Milnor arcs of $f= (x-y^2)((x-y^2)(y^2+1)-1)$}\label{fig:ex2}
\end{center}
\end{figure}
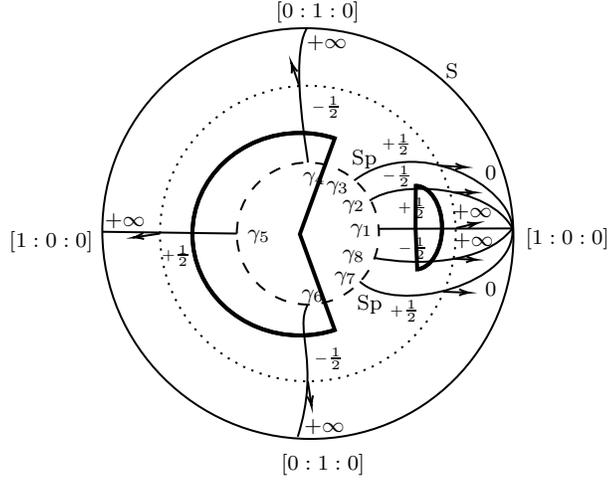

One has $d = d_{Re}=6$, and $\cL_{f} = \{p\}$ where $p:=[1:0:0]$, with $d_{p} =6$ (cf Definition \ref{d:dp}).
In the chart $\{x=1\}$ of $\bP^2$,  the germ at $p$ of $\overline{M(f)}$ is defined by the equation:
$$\hat{h}(y,z)=   y (6 y^4 - 8 y^2 z + 2 y^4 z + 2 z^2 + 2 y^2 z^2 - 4 z^3 + 
   2 y^2 z^3 + z^5) =0,$$
and therefore $\Cone \overline{M_{\bC}(f)}_{p}= L^\ity \cup\{y=0\}$.

The polynomial $f$ has a global minimum at the point $(\frac{1}{2},0)\in \bR^{2}$, with critical value $-\frac14$.  The fibre of $f$ is empty over the interval $]-\infty, -\frac14[$. Over 
$[-\frac{1}{4}, 0[$, the fibre of $f$ is compact and connected, having  two arcs outside the disk $D_{R}$, one of which is splitting along the cluster $\{\gamma_7\}$, and the other is splitting along the cluster $\{\gamma_3\}$; both Milnor clusters are tangent to $L^\ity$  at the point $p$.  
Over the interval $]0, +\infty[$, the fibre of $f$  has two connected components: one of them corresponds to the cluster $\{\gamma_8,\gamma_1,\gamma_2\}$, and is vanishing\footnote{The cluster $\{\gamma_8,\gamma_1,\gamma_2\}$ may be contrasted to  Proposition \ref{p:tangency} in which such a situation cannot happen for a cluster associated to a finite limit value of $f$.} at the point $p$ with the value of  $f$ tending to $+\ity$.   The other component corresponds to the  vanishing cluster $\{\gamma_4, \gamma_5, \gamma_6\}$  and covers the entire line at infinity $L^\ity$ as $t\to +\ity$.  
 
 By direct computations we see that the germ $\overline{M_{\bC}(f)_{p}}$  has 3 non-singular branches: one is $\{y=0\}$ and contains the Milnor arcs  $\gamma_1$ and  $\gamma_5$. The other two branches  are tangent to the line at infinity\footnote{We get $\mult (\overline{M_{\bC}(f)_{p}}, L^{\ity}) = 1+2+2 =5$,  each tangency producing multiplicity 2 in \eqref{eq:Mil}.} and have both their two arcs on the same side of it, namely $\gamma_8$ with $\gamma_2$, and $\gamma_7$ with $\gamma_3$, respectively. There is a single Milnor arc on the other side of $L^\ity$, which fact may be contrasted to  Lemma \ref{l:d_{p}}(b),(c) about index gaps, see also Remark \ref{r:situat} and the computation of sign gaps in the proof of Theorem \ref{t:upbound}.

   By Theorem \ref{t:main1} we get $\ind_{\ity}(f)= 1+ 2\frac12 -2\frac12 =1$.   We have    $\deg S_{p}= 2$ because of the two non-singular real  branches of $\overline{M_{\bC}(f)}_{p}$ which are tangent to $L^{\ity}$, no non-real branches $K_{p}=\emptyset$, and   $\deg (R_{p}^{\red})=2$.  The inequality \eqref{eq:upbound} of Theorem \ref{t:upbound} then reads:
    \[\ind_{\ity}(f) \le 1+6 - 2 - 2 = 3, \] 
 since  the ``sign gaps'' of this formula count for zero in our case. The same inequality $\ind_{\ity}(f) \le 3$ is provided by Theorem \ref{t:l=1}(b).

 Nevertheless, if we consider the sharper estimation of the sign gap in the proof of Theorem \ref{t:upbound}, namely 
 $\frac12 (\lfloor  \frac{r_{p}}{2}\rfloor +   \lfloor  \frac{s_{p}}{2}  \rfloor )$, and since in our case we have $r_{p}=5$ and $s_{p} =1$, we  get a sign gap of 1.
 With this extra gap, we then obtain: 
 \[ \ind_{\ity}(f) \le 1 +6-2-2 -1 =2,\] 
which is indeed a better estimation,  still away by 1 from the actual index $\ind_{\ity}(f) =1$ of this example as computed above via Theorem \ref{t:main1}.  
 \end{example}

\begin{example} \label{s:exa4}
Let  $f(x,y)=x^2 + (xy-1)^2$.  
One has $d=d_{Re}=4$, and $\cL_f=\{ p,q\}$ where $p:=[1:0:0]$  and $q:=[0:1:0]$, with $d_p=d_q=2$.
Note that $f$ has empty fibres over $]-\ity, 0[$.  

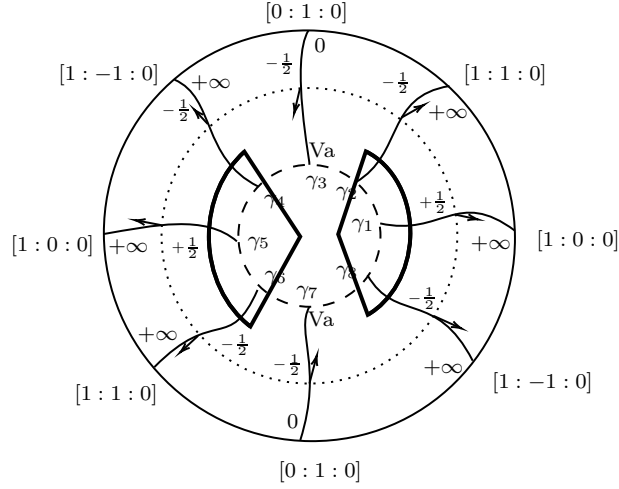
\begin{figure}[ht!] 
\begin{center}

\tikzset{every picture/.style={line width=0.75pt}} 

\begin{tikzpicture}[x=0.60pt,y=0.60pt,yscale=-1,xscale=1]

\draw   (329.47,29.92) .. controls (401.03,30.55) and (459.31,89.06) .. (459.63,160.62) .. controls (459.96,232.17) and (402.21,289.67) .. (330.64,289.05) .. controls (259.08,288.43) and (200.8,229.91) .. (200.47,158.36) .. controls (200.15,86.8) and (257.9,29.3) .. (329.47,29.92) -- cycle ;
\draw  [dash pattern={on 4.5pt off 4.5pt}] (285.26,159.49) .. controls (285.26,134.75) and (305.32,114.69) .. (330.05,114.69) .. controls (354.79,114.69) and (374.85,134.75) .. (374.85,159.49) .. controls (374.85,184.22) and (354.79,204.28) .. (330.05,204.28) .. controls (305.32,204.28) and (285.26,184.22) .. (285.26,159.49) -- cycle ;
\draw  [dash pattern={on 0.84pt off 2.51pt}] (236.83,159.49) .. controls (236.83,108) and (278.57,66.26) .. (330.05,66.26) .. controls (381.54,66.26) and (423.28,108) .. (423.28,159.49) .. controls (423.28,210.97) and (381.54,252.71) .. (330.05,252.71) .. controls (278.57,252.71) and (236.83,210.97) .. (236.83,159.49) -- cycle ;
\draw [color={rgb, 255:red, 0; green, 0; blue, 0 }  ,draw opacity=1 ]   (330.05,114.69) .. controls (331,115.67) and (317.67,49) .. (329.47,29.92) ;
\draw [color={rgb, 255:red, 0; green, 0; blue, 0 }  ,draw opacity=1 ]   (245,61) .. controls (271,72.33) and (261.67,110.33) .. (297.67,128.33) ;
\draw [color={rgb, 255:red, 0; green, 0; blue, 0 }  ,draw opacity=1 ]   (200.47,158.36) .. controls (231.67,153) and (264.33,145.67) .. (284.33,163) ;
\draw [color={rgb, 255:red, 0; green, 0; blue, 0 }  ,draw opacity=1 ]   (232.33,242.33) .. controls (270.33,203.67) and (279.67,233.67) .. (297.33,193.67) ;
\draw [color={rgb, 255:red, 0; green, 0; blue, 0 }  ,draw opacity=1 ]   (324.33,288.33) .. controls (339,241.67) and (321,220.33) .. (330.05,204.28) ;
\draw [color={rgb, 255:red, 0; green, 0; blue, 0 }  ,draw opacity=1 ]   (433,238.33) .. controls (409.67,193) and (384.33,215) .. (366.67,184.33) ;
\draw [color={rgb, 255:red, 0; green, 0; blue, 0 }  ,draw opacity=1 ]   (459.67,156.33) .. controls (423,129) and (413.67,160.33) .. (374.67,151.67) ;
\draw [color={rgb, 255:red, 0; green, 0; blue, 0 }  ,draw opacity=1 ]   (418.33,65) .. controls (377.67,71.67) and (389.67,107.67) .. (359.33,125.67) ;
\draw    (266,89) -- (258.57,81.06) ;
\draw [shift={(257.2,79.6)}, rotate = 46.89] [color={rgb, 255:red, 0; green, 0; blue, 0 }  ][line width=0.75]    (6.56,-1.97) .. controls (4.17,-0.84) and (1.99,-0.18) .. (0,0) .. controls (1.99,0.18) and (4.17,0.84) .. (6.56,1.97)   ;
\draw    (236.8,152.8) -- (221.96,149.97) ;
\draw [shift={(220,149.6)}, rotate = 10.78] [color={rgb, 255:red, 0; green, 0; blue, 0 }  ][line width=0.75]    (6.56,-1.97) .. controls (4.17,-0.84) and (1.99,-0.18) .. (0,0) .. controls (1.99,0.18) and (4.17,0.84) .. (6.56,1.97)   ;
\draw    (259.87,222.6) -- (249.87,231.84) ;
\draw [shift={(248.4,233.2)}, rotate = 317.25] [color={rgb, 255:red, 0; green, 0; blue, 0 }  ][line width=0.75]    (6.56,-1.97) .. controls (4.17,-0.84) and (1.99,-0.18) .. (0,0) .. controls (1.99,0.18) and (4.17,0.84) .. (6.56,1.97)   ;
\draw    (330.05,252.71) -- (334.29,236.73) ;
\draw [shift={(334.8,234.8)}, rotate = 104.84] [color={rgb, 255:red, 0; green, 0; blue, 0 }  ][line width=0.75]    (6.56,-1.97) .. controls (4.17,-0.84) and (1.99,-0.18) .. (0,0) .. controls (1.99,0.18) and (4.17,0.84) .. (6.56,1.97)   ;
\draw    (408.47,210.07) -- (424.18,217.18) ;
\draw [shift={(426,218)}, rotate = 204.35] [color={rgb, 255:red, 0; green, 0; blue, 0 }  ][line width=0.75]    (6.56,-1.97) .. controls (4.17,-0.84) and (1.99,-0.18) .. (0,0) .. controls (1.99,0.18) and (4.17,0.84) .. (6.56,1.97)   ;
\draw    (421.65,146.04) -- (434.44,148.78) ;
\draw [shift={(436.4,149.2)}, rotate = 192.09] [color={rgb, 255:red, 0; green, 0; blue, 0 }  ][line width=0.75]    (6.56,-1.97) .. controls (4.17,-0.84) and (1.99,-0.18) .. (0,0) .. controls (1.99,0.18) and (4.17,0.84) .. (6.56,1.97)   ;
\draw    (388.05,84.98) -- (400.66,77.79) ;
\draw [shift={(402.4,76.8)}, rotate = 150.32] [color={rgb, 255:red, 0; green, 0; blue, 0 }  ][line width=0.75]    (6.56,-1.97) .. controls (4.17,-0.84) and (1.99,-0.18) .. (0,0) .. controls (1.99,0.18) and (4.17,0.84) .. (6.56,1.97)   ;
\draw    (324.05,66.71) -- (320.85,80.45) ;
\draw [shift={(320.4,82.4)}, rotate = 283.11] [color={rgb, 255:red, 0; green, 0; blue, 0 }  ][line width=0.75]    (6.56,-1.97) .. controls (4.17,-0.84) and (1.99,-0.18) .. (0,0) .. controls (1.99,0.18) and (4.17,0.84) .. (6.56,1.97)   ;
\draw  [draw opacity=0][line width=1.5]  (292.86,216.73) .. controls (277.09,204.66) and (266.67,183.76) .. (266.67,160) .. controls (266.67,138.2) and (275.44,118.81) .. (289.06,106.44) -- (324.26,160) -- cycle ; \draw  [line width=1.5]  (292.86,216.73) .. controls (277.09,204.66) and (266.67,183.76) .. (266.67,160) .. controls (266.67,138.2) and (275.44,118.81) .. (289.06,106.44) ;  
\draw  [draw opacity=0][line width=1.5]  (365.63,106.25) .. controls (382.92,114.72) and (394.65,136.02) .. (393.62,160.43) .. controls (392.7,182.18) and (381.88,200.62) .. (366.84,209.31) -- (348.19,158.51) -- cycle ; \draw  [line width=1.5]  (365.63,106.25) .. controls (382.92,114.72) and (394.65,136.02) .. (393.62,160.43) .. controls (392.7,182.18) and (381.88,200.62) .. (366.84,209.31) ;  

\draw (331.47,33.32) node [anchor=north west][inner sep=0.75pt]  [font=\tiny]  {$0$};
\draw (314.13,270) node [anchor=north west][inner sep=0.75pt]  [font=\tiny]  {$0$};
\draw (253.47,56.07) node [anchor=north west][inner sep=0.75pt]  [font=\tiny]  {$+\infty $};
\draw (403,75) node [anchor=north west][inner sep=0.75pt]  [font=\tiny]  {$+\infty $};
\draw (201.47,160) node [anchor=north west][inner sep=0.75pt]  [font=\tiny]  {$+\infty $};
\draw (220,215) node [anchor=north west][inner sep=0.75pt]  [font=\tiny]  {$+\infty $};
\draw (400,236.07) node [anchor=north west][inner sep=0.75pt]  [font=\tiny]  {$+\infty $};
\draw (430,158.07) node [anchor=north west][inner sep=0.75pt]  [font=\tiny]  {$+\infty $};
\draw (355,146.73) node [anchor=north west][inner sep=0.75pt]  [font=\tiny]  {$\gamma _{1}$};
\draw (345,126.73) node [anchor=north west][inner sep=0.75pt]  [font=\tiny]  {$\gamma _{2}$};
\draw (326,120) node [anchor=north west][inner sep=0.75pt]  [font=\tiny]  {$\gamma _{3}$};
\draw (299.67,131.73) node [anchor=north west][inner sep=0.75pt]  [font=\tiny]  {$\gamma _{4}$};
\draw (289.33,158.07) node [anchor=north west][inner sep=0.75pt]  [font=\tiny]  {$\gamma _{5}$};
\draw (300,180) node [anchor=north west][inner sep=0.75pt]  [font=\tiny]  {$\gamma _{6}$};
\draw (320,189) node [anchor=north west][inner sep=0.75pt]  [font=\tiny]  {$\gamma _{7}$};
\draw (345,176.73) node [anchor=north west][inner sep=0.75pt]  [font=\tiny]  {$\gamma _{8}$};
\draw (337.83,105) node  [font=\tiny] [align=left] {\begin{minipage}[lt]{9.75pt}\setlength\topsep{0pt}
Va
\end{minipage}};
\draw (337.22,212.11) node  [font=\tiny] [align=left] {\begin{minipage}[lt]{9.75pt}\setlength\topsep{0pt}
Va
\end{minipage}};
\draw (308.67,300.07) node [anchor=north west][inner sep=0.75pt]  [font=\tiny]  {$[0:1:0]$};
\draw (471.33,154.07) node [anchor=north west][inner sep=0.75pt]  [font=\tiny]  {$[1:0:0]$};
\draw (425.33,49.4) node [anchor=north west][inner sep=0.75pt]  [font=\tiny]  {$[ 1:1:0]$};
\draw (443.33,243.4) node [anchor=north west][inner sep=0.75pt]  [font=\tiny]  {$[ 1:-1:0]$};
\draw (397.33,130.07) node [anchor=north west][inner sep=0.75pt]  [font=\fontsize{0.44em}{0.53em}\selectfont]  {$+\frac{1}{2}$};
\draw (373.33,54.73) node [anchor=north west][inner sep=0.75pt]  [font=\fontsize{0.44em}{0.53em}\selectfont]  {$-\frac{1}{2}$};
\draw (300,42.73) node [anchor=north west][inner sep=0.75pt]  [font=\fontsize{0.44em}{0.53em}\selectfont]  {$-\frac{1}{2}$};
\draw (235,72.07) node [anchor=north west][inner sep=0.75pt]  [font=\fontsize{0.44em}{0.53em}\selectfont]  {$-\frac{1}{2}$};
\draw (241.33,158.73) node [anchor=north west][inner sep=0.75pt]  [font=\fontsize{0.44em}{0.53em}\selectfont]  {$+\frac{1}{2}$};
\draw (272,220.07) node [anchor=north west][inner sep=0.75pt]  [font=\fontsize{0.44em}{0.53em}\selectfont]  {$-\frac{1}{2}$};
\draw (305,232.47) node [anchor=north west][inner sep=0.75pt]  [font=\fontsize{0.44em}{0.53em}\selectfont]  {$-\frac{1}{2}$};
\draw (390,190.73) node [anchor=north west][inner sep=0.75pt]  [font=\fontsize{0.44em}{0.53em}\selectfont]  {$-\frac{1}{2}$};
\draw (300,10) node [anchor=north west][inner sep=0.75pt]  [font=\tiny]  {$[0:1:0]$};
\draw (140,157.4) node [anchor=north west][inner sep=0.75pt]  [font=\tiny]  {$[1:0:0]$};
\draw (180,251.4) node [anchor=north west][inner sep=0.75pt]  [font=\tiny]  {$[ 1:1:0]$};
\draw (170,49.4) node [anchor=north west][inner sep=0.75pt]  [font=\tiny]  {$[ 1:-1:0]$};

\end{tikzpicture}

\end{center}
\caption{The Milnor arcs of $f= x^2 + (xy-1)^2$}\label{f:exam4}
\end{figure}

\bigskip
The Milnor set $M(f)$ is defined by the equation $x^2 + x y - x^3 y - y^2 + x y^3=0$.
At $q \in L^\ity$ there are two clusters with a single Milnor arc, namely $\{\gamma_3\}$ and $\{\gamma_7\}$,  and the corresponding fibre components are both vanishing  at the value $0$.  There are two more clusters, namely $\{\gamma_8,\gamma_1,\gamma_2\}$, and $\{\gamma_4,\gamma_5,\gamma_6\}$,
the fibre components of which are both tending to $+\ity$ and cover half the circle $S$ each of them. 

By  Theorem \ref{t:main1} we get: 
$$\ind_{\ity}(f) =  1+ \frac12 \sum \Sp(p, \lambda) - \frac12 \sum\Va(p, \lambda)  - \frac12  \Va(\pm \infty) =1 +  \frac12 \cdot 0 - \frac12 \cdot  2 - \frac12 \cdot  2=-1,$$
where the sums are over $\{p\in \cL_{f}, \lambda\in \bR\}$.

We have $\Cone \overline{M_{\bC}(f)}_q=\{x=0\}$ and $\Cone \overline{M_{\bC}(f)}_p=\{y=0\}$, with multiplicity 1, and the Milnor set germs at $p$ and $q$ are non-singular and transversal to $L^\ity_{\bC}$.
Therefore all the sets $S_{p},  K_p, S_{q}$ and $K_q$ are empty, and $\deg R^{\red}_{q}=\deg R^{\red}_{p}=1$. Then  Theorem \ref{t:upbound} and Corollary \ref{c:upbound2}, with  \eqref{eq:upbound2} and \eqref{eq:upbound3},  yield all the same bound: 
 \[ \ind_{\ity}(f) \le 1+ d_{Re} - |\cL_{f}| = 1+4-4=1. \]

By considering the genuine sign gaps as in Lemma \ref{l:d_{p}}(a),  one actually obtains a gap of 2  at the point $q \in \cL_f$ due to the two Milnor arcs with index $-\frac12$ at the value 0 of $f$.  We then get
\[ \ind_{\ity}(f) \le 1+4-4-2=-1,\]
 which coincides with the actual index at infinity of $f$ computed above.

\end{example}

\bigskip


\end{document}